\documentclass[12pt]{amsart}
\usepackage[margin=2cm]{geometry}
\usepackage{amsmath, amsthm, amssymb}
\usepackage{epsfig}
\usepackage{rawfonts}
\usepackage{enumerate}
\usepackage{graphics}
\usepackage{multirow}
\usepackage{xspace}
\usepackage{graphicx}
\usepackage{pgf,tikz,pgfplots}
\usepackage{mathrsfs}
\usetikzlibrary{arrows}
\usepackage{amsmath}
\usepackage{amsfonts}
\usepackage{amssymb}
\usepackage{amsthm}
\usepackage{graphicx}
\usepackage{booktabs}
\usepackage{caption}
\usepackage{listings}
\usepackage{setspace}
\usepackage[mathscr]{eucal}
\usepackage{pgfplots}
\usepackage{hyperref}
\usepackage{wrapfig}
\usepackage{floatflt,epsfig}
\usepackage{dsfont}
\usepackage{amscd}
\usepackage{tikz-cd}
\usepackage{fancyhdr}
\usepackage[all]{xy}
\usepackage{latexsym}
\usepackage{amscd}
\usepackage{pifont}
\usepackage{eufrak}
\usepackage{subfig}

\def\ZZ{{\mathbb Z}}

\newcommand{\cC}{\mathcal{C}}

\newcommand{\cP}{\mathcal{P}}
\newcommand{\cH}{\mathcal{H}}
\newcommand{\cI}{\mathcal{I}}
\newcommand{\cK}{\mathcal{K}}
\newcommand{\cB}{\mathcal{B}}
\newcommand{\cM}{\mathcal{M}}
\newcommand{\cF}{\mathcal{F}}
\newcommand{\cQ}{\mathcal{Q}}
\newcommand{\cW}{\mathcal{W}}
\newcommand{\cR}{\mathcal{R}}
\newcommand{\cS}{\mathcal{S}}

\newcommand{\cL}{\mathcal{L}}
\newcommand{\cV}{\mathcal{V}}

\newcommand{\rHP}{\mathrm{HP}}

\theoremstyle{plain}
\newtheorem{theorem}{Theorem}[section]
\newtheorem{proposition}[theorem]{Proposition}%
\newtheorem{lemma}[theorem]{Lemma}%
\newtheorem{coro}[theorem]{Corollary}%

\newtheorem{remark}[theorem]{Remark}%

\newtheorem{definition}[theorem]{Definition}%

\begin{document}
		
		\title[Hilbert series and Gorenstein property for some non-simple polyominoes]{HILBERT-POINCAR\'E SERIES AND GORENSTEIN PROPERTY FOR SOME NON-SIMPLE POLYOMINOES}
		
		\author{CARMELO CISTO}
		\address{Universit\'{a} di Messina, Dipartimento di Scienze Matematiche e Informatiche, Scienze Fisiche e Scienze della Terra\\
			Viale Ferdinando Stagno D'Alcontres 31\\
			98166 Messina, Italy}
		\email{carmelo.cisto@unime.it}

		\author{FRANCESCO NAVARRA}
		\address{Universit\'{a} di Messina, Dipartimento di Scienze Matematiche e Informatiche, Scienze Fisiche e Scienze della Terra\\
			Viale Ferdinando Stagno D'Alcontres 31\\
			98166 Messina, Italy}
		\email{francesco.navarra@unime.it}

		\author{ROSANNA UTANO}
		\address{Universit\'{a} di Messina, Dipartimento di Scienze Matematiche e Informatiche, Scienze Fisiche e Scienze della Terra\\
			Viale Ferdinando Stagno D'Alcontres 31\\
			98166 Messina, Italy}
		\email{rosanna.utano@unime.it}

		\subjclass[2010]{05B50, 05E40, 13C05, 13G05}		
		
		

    \keywords{Polyominoes, Hilbert-Poincar\'e series, rook polynomial, Gorenstein.}
		

	\begin{abstract}
	Let $\mathcal{P}$ be a closed path having no zig-zag walks, a kind of non-simple thin polyomino. In this paper we give a combinatorial interpretation of the $h$-polynomial of $K[\mathcal{P}]$, showing that it is the rook polynomial of $\mathcal{P}$. It is known by Rinaldo and Romeo (2021), that if $\mathcal{P}$ is a simple thin polyomino then the $h$-polynomial is equal to the rook polynomial of $\mathcal{P}$ and it is conjectured that this property characterizes all thin polyominoes. Our main demonstrative strategy is to compute the reduced Hilbert-Poincar\'e series of the coordinate ring attached to a closed path $\mathcal{P}$ having no zig-zag walks, as a combination of the Hilbert-Poincar\'e series of convenient simple thin polyominoes. As a consequence we prove that the Krull dimension is equal to $\vert V(\mathcal{P})\vert -\mathrm{rank}\, \mathcal{P}$ and the regularity of $K[\mathcal{P}]$ is the rook number of $\mathcal{P}$. Finally we characterize the Gorenstein prime closed paths, proving that $K[\mathcal{P}]$ is Gorenstein if and only if $\mathcal{P}$ consists of maximal blocks of length three.
\end{abstract}

	\maketitle
		
	\section{Introduction}
	
	\noindent A \textit{polyomino} is a finite collection of unitary squares joined edge by edge. In 2012 A.A. Qureshi defined the \textit{polyomino ideal} attached to a polyomino $\cP$ as the ideal $I_\cP$ generated by all inner 2-minors of $\cP$ in the polynomial ring over $K$ in the variables $x_v$, for each $v$ vertex of $\cP$. For more details see \cite{Qureshi}.\\
	\noindent The study of the main algebraic properties of the quotient ring $K[\cP]=S/I_{\cP}$ depending on the shape of $\cP$ has become an exciting line of research.
	For instance, several mathematicians have studied the primality of $I_\cP$, normality and Cohen-Macaulayness of $K[\cP]$. For some references to these results we mention  \cite{Cisto_Navarra_closed_path}, \cite{Cisto_Navarra_weakly}, \cite{Cisto_Navarra_CM_closed_path},  \cite{Dinu_Navarra}, \cite{Simple equivalent balanced}, \cite{def balanced}, \cite{Not simple with localization}, \cite{Trento}, \cite{Trento2}, \cite{Simple are prime}, \cite{Shikama}. We mention also other references, which provided inspiration for this work. 
	In \cite{Andrei} the author classifies all convex polyominoes whose coordinate rings are Gorenstein and computes the regularity of the coordinate ring of any stack polyomino in terms of the smallest interval which contains its
	vertices. In \cite{L-convessi} the authors give a new combinatorial interpretation of the regularity of the coordinate ring attached to an $L$-convex polyomino, as the rook number of $\cP$, that is the maximum number of rooks which can be arranged in $\cP$ in non-attacking positions. In \cite{Trento3} it is showed that if $\cP$ is a simple thin polyomino, which is a polyomino not containing the square tetromino, then the $h$-polynomial $h(t)$ of $K[\cP]$ is the rook polynomial $r_{\cP}(t)=\sum_{i=0}^n r_i x^i$ of $\cP$, whose coefficient $r_i$ represents the number of distinct possibilities of arranging $i$ rooks on cells of $\cP$ in non-attacking positions (with the convention $r_0=1$). Gorenstein simple thin polyominoes are also characterized using the $S$-property and finally it is conjectured that a polyomino is thin if and only if $h(t)=r_{\cP}(t)$. In this paper we give also a partial support to this conjecture, since we provide an affirmative answer for a particular class of non-simple thin polyominoes, namely closed paths. In \cite{Kummini rook polynomial} this conjecture is discussed for a certain	class of polyominoes. In a recent paper \cite{Parallelogram Hilbert series}  the authors introduce a particular equivalence relation on the rook complex of a simple polyomino and they conjecture that the number of equivalence classes of $i$ non-attacking rooks arrangements
	is exactly the $i$-th coefficient of the $h$-polynomial in the reduced Hilbert-Poincar\'e series. Moreover they prove it for the class of parallelogram polyominoes and, by a computational method, also for all simple polyominoes having rank at most eleven.\\
Let $\cP$ be a closed path without zig-zag walks, which is a particular non-simple thin polyomino introduced in \cite{Cisto_Navarra_closed_path}. The aim of this paper is to provide a combinatorial view of the $h$-polynomial of $K[\cP]$. In this direction, we show that it is equal to the rook polynomial of $\cP$, studying the Hilbert-Poincar\'e series of certain non-simple polyominoes having just one hole and relating them to the Hilbert-Poincar\'e series of simple thin polyominoes included in $\cP$. As a consequence we compute the Krull dimension and the regularity of $K[\cP]$ and, in conclusion, we give a characterization of the Gorenstein property in terms of the length of maximal blocks of $\cP$. Roughly speaking, a closed path is a sequence of cells, similar to a pearl necklace on a table, in which each pearl corresponds to a cell, producing only one hole. This class of polyominoes is defined in \cite{Cisto_Navarra_closed_path}, where the authors characterize their primality by the non-existence of zig-zag walks, equivalently by the existence of an $L$-configuration or a ladder of at least three steps.
    In Section \ref{Section: Introduction} we introduce definitions and notations, which are fundamental along the paper. In Section \ref*{Section:Poincare-Hilbert series of certain non-simple polyominoes}, we introduce a particular polyomino $\cL$ and define the class of $(\cL,\cC)$-polyominoes, $\cC$ a generic polyomino, and we provide an explicit formula for the Hilbert-Poincar\'e series of the related coordinate rings, depending on the Hilbert-Poincar\'e series of some polyominoes obtained by eliminating specific cells. In a particular case we compute also the Krull dimension of the coordinate ring of a polyomino belonging to this class. In Section \ref{Section: H-P series of prime closed path polyominoes having no L-configuration} we assume that $\cP$ is a closed path polyomino and we deal the case in which $\cP$ has no $L$-configuration but it contains a ladder of at least three steps.  
    In order to reach our aim we describe the initial ideals of some subpolyominoes of $\cP$ with respect to some monomial orders. The case $\cP$ has an $L$-configuration is discussed implicitly in Section 3, since such a closed path is a particular $(\mathcal{L},\cC)$-polyomino, with $\cC$ a simple path polyomino. In this way we fulfil the study of the Hilbert-Poincar\'e series for the class of closed paths without zig-zag walks. 
In Section 5 we prove that the $h$-polynomial of $K[\cP]$, where $\cP$ is a prime closed path polyomino, is the rook polynomial of $\cP$, obtaining as a consequence the regularity and the Krull dimension of $K[\cP]$. 
Finally we characterize all Gorenstein prime closed paths using the $S$-property. We conclude giving some open questions. 

\section{Polyominoes and polyomino ideals}\label{Section: Introduction}

\noindent Let $(i,j),(k,l)\in \mathbb{Z}^2$. We say that $(i,j)\leq(k,l)$ if $i\leq k$ and $j\leq l$. Consider $a=(i,j)$ and $b=(k,l)$ in $\mathbb{Z}^2$ with $a\leq b$. The set $[a,b]=\{(m,n)\in \mathbb{Z}^2: i\leq m\leq k,\ j\leq n\leq l \}$ is called an \textit{interval} of $\mathbb{Z}^2$. 
In addition, if $i< k$ and $j<l$ then $[a,b]$ is a \textit{proper} interval. In such a case we say $a, b$ the \textit{diagonal corners} of $[a,b]$ and $c=(i,l)$, $d=(k,j)$ the \textit{anti-diagonal corners} of $[a,b]$. If $j=l$ (or $i=k$) then $a$ and $b$ are in \textit{horizontal} (or \textit{vertical}) \textit{position}. We denote by $]a,b[$ the set $\{(m,n)\in \mathbb{Z}^2: i< m< k,\ j< n< l\}$. A proper interval $C=[a,b]$ with $b=a+(1,1)$ is called a \textit{cell} of $\ZZ^2$; moreover, the elements $a$, $b$, $c$ and $d$ are called respectively the \textit{lower left}, \textit{upper right}, \textit{upper left} and \textit{lower right} \textit{corners} of $C$. The sets $\{a,c\}$, $\{c,b\}$, $\{b,d\}$ and $\{a,d\}$ are the \textit{edges} of $C$. We put $V(C)=\{a,b,c,d\}$ and $E(C)=\{\{a,c\},\{c,b\},\{b,d\},\{a,d\}\}$. Let $\cS$ be a non-empty collection of cells in $\mathbb{Z}^2$. The set of the vertices and of the edges of $\cS$ are respectively $V(\cS)=\bigcup_{C\in \cS}V(C)$ and $E(\cS)=\bigcup_{C\in \cS}E(C)$, $\mathrm{rank}\,\cS$ is the number of cells belonging to $\cS$. If $C$ and $D$ are two distinct cells of $\cS$, then a \textit{walk} from $C$ to $D$ in $\cS$ is a sequence $\cC:C=C_1,\dots,C_m=D$ of cells of $\ZZ^2$ such that $C_i \cap C_{i+1}$ is an edge of $C_i$ and $C_{i+1}$ for $i=1,\dots,m-1$. In addition, if $C_i \neq C_j$ for all $i\neq j$, then $\cC$ is called a \textit{path} from $C$ to $D$. We say that $C$ and $D$ are connected in $\cS$ if there exists a path of cells in $\cS$ from $C$ to $D$.
A \textit{polyomino} $\cP$ is a non-empty, finite collection of cells in $\mathbb{Z}^2$ where any two cells of $\cP$ are connected in $\cP$. For instance, see Figure \ref{Figura: Polimino introduzione}.
\begin{figure}[h]
	\centering
	\includegraphics[scale=0.5]{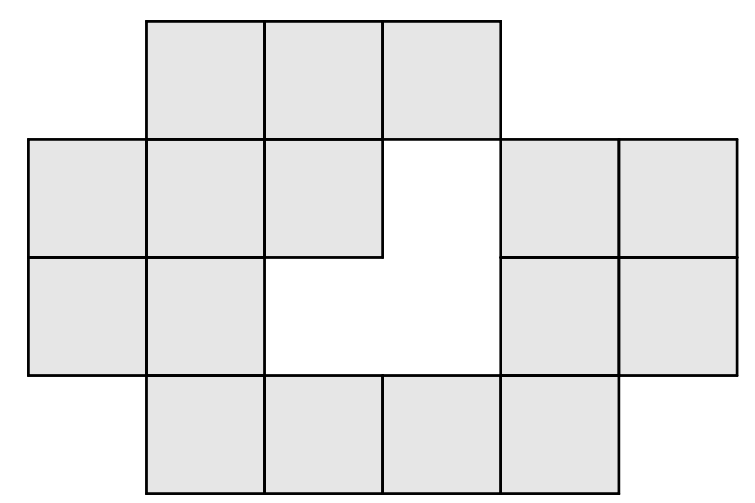}
	\caption{A polyomino.}
	\label{Figura: Polimino introduzione}
\end{figure}

\noindent We say that a polyomino $\cP$ is \textit{simple} if for any two cells $C$ and $D$ not in $\cP$ there exists a path of cells not in $\cP$ from $C$ to $D$. A finite collection of cells $\cH$ not in $\cP$ is a \textit{hole} of $\cP$ if any two cells of $\cH$ are connected in $\cH$ and $\cH$ is maximal with respect to set inclusion. For example, the polyomino in Figure \ref{Figura: Polimino introduzione} is not simple with an hole. Obviously, each hole of $\cP$ is a simple polyomino and $\cP$ is simple if and only if it has not any hole.\\
Let $A$ and $B$ be two cells of $\mathbb{Z}^2$, with $a=(i,j)$ and $b=(k,l)$ as the lower left corners of $A$ and $B$ and $a\leq b$. A \textit{cell interval} $[A,B]$ is the set of the cells of $\mathbb{Z}^2$ with lower left corner $(r,s)$ such that $i\leqslant r\leqslant k$ and $j\leqslant s\leqslant l$. If $(i,j)$ and $(k,l)$ are in horizontal (or vertical) position, we say that the cells $A$ and $B$ are in \textit{horizontal} (or \textit{vertical}) \textit{position}.\\
Let $\cP$ be a polyomino, $A$ and $B$ two cells of $\cP$ in vertical or horizontal position. 	 
The cell interval $[A,B]$, containing $n>1$ cells, is called a
\textit{block of $\cP$ of rank n} if all cells of $[A,B]$ belong to $\cP$. The cells $A$ and $B$ are called \textit{extremal} cells of $[A,B]$. Moreover, a block $\cB$ of $\cP$ is \textit{maximal} if there does not exist any block of $\cP$ which contains properly $\cB$. It is clear that an interval of $\ZZ^2$ identifies a cell interval of $\ZZ^2$ and vice versa, hence we can associate to an interval $I$ of $\ZZ^2$ the corresponding cell interval denoted by $\cP_{I}$. A proper interval $[a,b]$ is called an \textit{inner interval} of $\cP$ if all cells of $\cP_{[a,b]}$ belong to $\cP$. We denote by $\cI(\cP)$ the set of all inner intervals of $\cP$. An interval $[a,b]$ with $a=(i,j)$, $b=(k,j)$ and $i<k$ is called a \textit{horizontal edge interval} of $\cP$ if the sets $\{(\ell,j),(\ell+1,j)\}$ are edges of cells of $\cP$ for all $\ell=i,\dots,k-1$. In addition, if $\{(i-1,j),(i,j)\}$ and $\{(k,j),(k+1,j)\}$ do not belong to $E(\cP)$, then $[a,b]$ is called a \textit{maximal} horizontal edge interval of $\cP$. We define similarly a \textit{vertical edge interval} and a \textit{maximal} vertical edge interval. \\
\noindent As in \cite{Trento} we call a \textit{zig-zag walk} of $\cP$ a sequence $\cW:I_1,\dots,I_\ell$ of distinct inner intervals of $\cP$ where, for all $i=1,\dots,\ell$, the interval $I_i$ has either diagonal corners $v_i$, $z_i$ and anti-diagonal corners $u_i$, $v_{i+1}$ or anti-diagonal corners $v_i$, $z_i$ and diagonal corners $u_i$, $v_{i+1}$, such that:
\begin{enumerate}[(1)]
	\item $I_1\cap I_\ell=\{v_1=v_{\ell+1}\}$ and $I_i\cap I_{i+1}=\{v_{i+1}\}$, for all $i=1,\dots,\ell-1$;
	\item $v_i$ and $v_{i+1}$ are on the same edge interval of $\cP$, for all $i=1,\dots,\ell$;
	\item for all $i,j\in \{1,\dots,\ell\}$ with $i\neq j$, there exists no inner interval $J$ of $\cP$ such that $z_i$, $z_j$ belong to $J$.
\end{enumerate}
\noindent In according to \cite{Cisto_Navarra_closed_path}, we recall the definition of a \textit{closed path polyomino}, and the configuration of cells characterizing its primality. We say that a polyomino $\cP$ is a \textit{closed path} if it is a sequence of cells $A_1,\dots,A_n, A_{n+1}$, $n>5$, such that:
\begin{enumerate}[(1)]
	\item $A_1=A_{n+1}$;
	\item $A_i\cap A_{i+1}$ is a common edge, for all $i=1,\dots,n$;
	\item $A_i\neq A_j$, for all $i\neq j$ and $i,j\in \{1,\dots,n\}$;
	\item for all $i\in\{1,\dots,n\}$ and for all $j\notin\{i-2,i-1,i,i+1,i+2\}$ then $V(A_i)\cap V(A_j)=\emptyset$, where $A_{-1}=A_{n-1}$, $A_0=A_n$, $A_{n+1}=A_1$ and $A_{n+2}=A_2$. 
\end{enumerate}
A path of five cells $C_1, C_2, C_3, C_4, C_5$ of $\cP$ is called an \textit{L-configuration} if the two sequences $C_1, C_2, C_3$ and $C_3, C_4, C_5$ go in two orthogonal directions. A set $\cB=\{\cB_i\}_{i=1,\dots,n}$ of maximal horizontal (or vertical) blocks of rank at least two, with $V(\cB_i)\cap V(\cB_{i+1})=\{a_i,b_i\}$ and $a_i\neq b_i$ for all $i=1,\dots,n-1$, is called a \textit{ladder of $n$ steps} if $[a_i,b_i]$ is not on the same edge interval of $[a_{i+1},b_{i+1}]$ for all $i=1,\dots,n-2$. For instance, in Figure \ref{Figura:L conf + Ladder}(a) there is a closed path having an $L$-configuration and a ladder of three steps. We recall that a closed path has no zig-zag walks if and only if it contains an $L$-configuration or a ladder of at least three steps (see \cite[Section 6]{Cisto_Navarra_closed_path}).
\\A finite non-empty collection of cells $\cP$ is called a \textit{weakly closed path} (\cite[Definition 4.1]{Cisto_Navarra_weakly}) if it is a path of $n$ cells $A_1,\dots,A_{n-1},A_n=A_0$ with $n>6$ such that:
	\begin{enumerate}[(1)]
		\item $\vert V(A_0)\cap V(A_1)\vert =1$;
		\item $V(A_2)\cap V(A_{0})=V(A_{n-1})\cap V(A_1)=\emptyset$;
		\item $V(A_i)\cap V(A_j)=\emptyset$ for all $i\in\{1,\dots,n\}$ and for all $j\notin\{i-2,i-1,i,i+1,i+2\}$, where the indices are reduced modulo $n$.
		\end{enumerate}
		A finite collection of cells of $\cP$, made up of a maximal horizontal (resp. vertical) block $[A,B]$ of $\cP$ of length at least two and two distinct cells $C$ and $D$ of $\cP$, not belonging to $[A,B]$, with $V(C)\cap V([A,B])=\{a_1\}$ and $V(D)\cap V([A,B])=\{a_2,b_2\}$, $a_2\neq b_2$, is called a \textit{weak ladder} if $[a_2,b_2]$ is not on the same maximal horizontal (resp. vertical) edge interval of $\cP$ containing $a_1$ (see Figure \ref{Figura:L conf + Ladder}(b))
		
\begin{figure}[h!]
	\centering
	\subfloat[]{\includegraphics[scale=0.55]{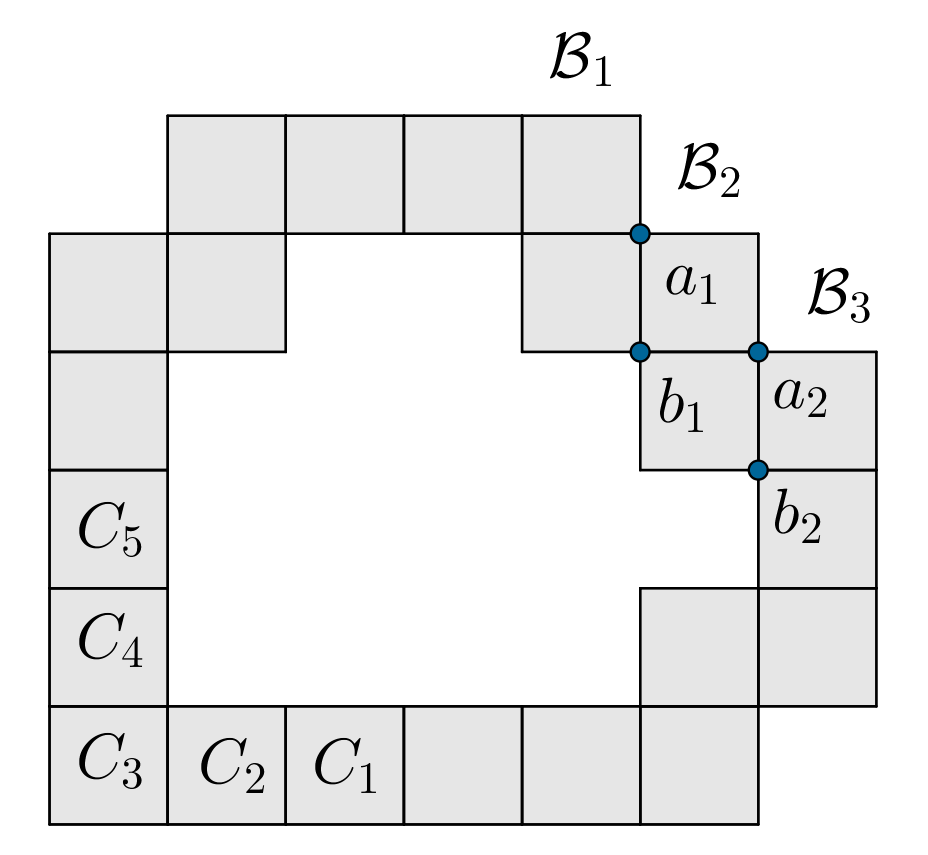}}
	\qquad\qquad
	\subfloat[]{\includegraphics[scale=0.65]{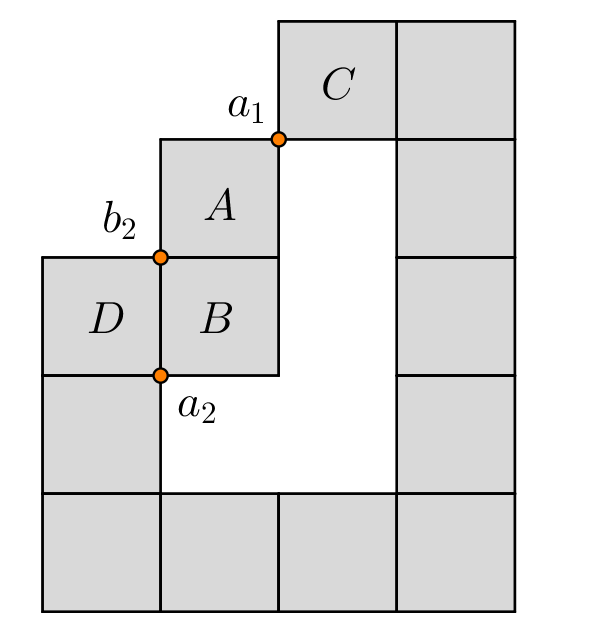}} 
	\caption{Examples of a closed path (a) and a weakly closed path (b).}
	\label{Figura:L conf + Ladder}
\end{figure}

\noindent Let $\cP$ be a polyomino. We set $S_\cP=K[x_v\vert  v\in V(\cP)]$, where $K$ is a field. If $[a,b]$ is an inner interval of $\cP$, with $a$, $b$ and $c$, $d$ respectively diagonal and anti-diagonal corners, then the binomial $x_ax_b-x_cx_d$ is called an \textit{inner 2-minor} of $\cP$. We call the ideal $I_{\cP}$ in $S_\cP$ generated by all the inner 2-minors of $\cP$ the \textit{polyomino ideal} of $\cP$. We set also $K[\cP] = S_\cP/I_{\cP}$, which is the \textit{coordinate ring} of $\cP$. \\
\noindent We recall some notions on the Hilbert function and the Hilbert-Poincar\'e series of a graded $K$-algebra $R/I$. Let $R$ be a graded $K$-algebra and $I$ be an homogeneous ideal of $R$. Then $R/I$ has a natural structure of graded $K$-algebra as $\bigoplus_{k\in\mathbb{N}}(R/I)_k$. The numerical function $\mathrm{H}_{R/I}:\mathbb{N}\to \mathbb{N}$ with $\mathrm{H}_{R/I}(k)=\dim_{K} (R/I)_k$ is called the \textit{Hilbert function} of $R/I$. The formal series $\rHP_{R/I}(t)=\sum_{k\in\mathbb{N}}\mathrm{H}_{R/I}(k)t^k$ is called the \textit{Hilbert-Poincar\'e series} of $R/I$. 
It is known by Hilbert-Serre theorem that there exists a polynomial $h(t)\in \mathbb{Z}[t]$ with $h(1)\neq0$ such that $\rHP_{R/I}(t)=\frac{h(t)}{(1-t)^d}$, where $d$ is the Krull dimension of $R/I$. Moreover, if $R/I$ is Cohen-Macaulay then $\mathrm{reg}(R/I)=\deg h(t)$. Recall also that if $S=K[x_1,\ldots,x_n]$ then $\rHP_{S}(t)=\frac{1}{(1-t)^n}$.\\
We will use frequently the following well known results (see for instance \cite[Chapter 5]{Villareal}).

\begin{proposition}\label{Proposizione: la prima che serve per calcolare HP}
	Let $R$ be a graded $K$-algebra and $I$ be a graded ideal of $R$. Let $q$ be an homogeneous element of $R$ of degree $m$.  Consider the exact sequence: $$0 \longrightarrow R/(I:q)  \longrightarrow R/I  \longrightarrow R/(I,q) \longrightarrow0  $$ Then $\rHP_{R/I}(t)=\rHP_{R/(I,q)}(t)+t^m\rHP_{R/(I:q)}(t)$.
\end{proposition}

\begin{proposition}\label{Hilber-tensoriale}
	Let $A$ and $B$ be standard graded $K$-algebras over a field $K$. Then $\rHP_{A \otimes_K B}(t)=\rHP_{A}(t) \cdot \rHP_{B}(t)$.
\end{proposition}

\begin{remark} \rm
	We will often use also the following elementary fact: if $\mathbf{X}$ is a set of indeterminates, $\mathbf{X}_1,\mathbf{X}_2\subset \mathbf{X}$ form a partition of $\mathbf{X}$ into disjoint non-empty subsets and $I$ is an ideal of $K[\mathbf{X}]$, $K$ a field, such that each generator of $I$ belongs to $K[\mathbf{X}_j]$ for $j\in \{1,2\}$, then $K[\mathbf{X}]/I\cong K[\mathbf{X}_1]/I_1 \otimes_K K[\mathbf{X}_2]/I_2$, where $I_j=I\cap K[\mathbf{X}_j]$ for $j\in \{1,2\}$ (see \cite[Proposition 3.1.33]{Villareal}). 
\end{remark} 	 

\noindent If $\cP$ is a polyomino then $\rHP_{K[\cP]}(t)=\frac{h(t)}{(1-t)^d}$ for some $h(t)\in \mathbb{Z}[t]$ with $h(1)\neq0$, so we denote $h(t)=h_{K[\cP]}(t)$ along the paper. Finally, if $n\in \mathbb{N}$ as usual we denote by $[n]$ the set $\{1,2,\ldots,n\}$.\\

%

\section{Hilbert-Poincar\'e series of certain non-simple polyominoes}\label{Section:Poincare-Hilbert series of certain non-simple polyominoes}

\noindent In this section we examine a particular class of non-simple polyominoes that we introduce in the following definition.

\begin{definition}\rm\label{Definizione: (L,C)-polimino}
	Let $\cL$ be the union of the two cell intervals $[A,A_r]$, consisting of the cells $A,A_1,\ldots ,A_r$, and $[A,B_s]$, consisting of the cells $A,B_1,\ldots,B_s$, where $A,A_r$ and $A,B_s$ are respectively in horizontal and vertical position with $r,s\geq 2$. We denote by $a,b$ and $c,d$ respectively the diagonal and anti-diagonal corners of $A$, by $d_i$ and $a_i$ respectively the upper left and upper right corners of $B_i$ for $i\in [s]$ and by $b_j$ and $c_j$ respectively the upper and lower right corners of $A_j$ for $j\in[r]$. Let $\cC$ be a polyomino.  We say that a polyomino $\cP$ is an $(\mathcal{L},\mathcal{C})$\textit{-polyomino} if $\cP=\mathcal{L}\sqcup \mathcal{C}$ and it satisfies one and only one of the following four conditions (see also Figure~\ref{Figura:esempio P(L,C)}):
	\begin{enumerate}[(1)]
		\item $V(\cL)\cap V(\mathcal{\cC})=\{a_{s-1},a_s,b_{r-1},b_r\}$;
		\item $V(\cL)\cap V(\mathcal{\cC})=\{a_{s-1},a_s,c_{r-1},c_r\}$;
		\item $V(\cL)\cap V(\mathcal{\cC})=\{d_{s-1},d_s,b_{r-1},b_r\}$;
		\item $V(\cL)\cap V(\mathcal{\cC})=\{d_{s-1},d_s,c_{r-1},c_r\}$;
	\end{enumerate}
	\begin{figure}[h]
		\centering
		\subfloat[]{\includegraphics[scale=0.6]{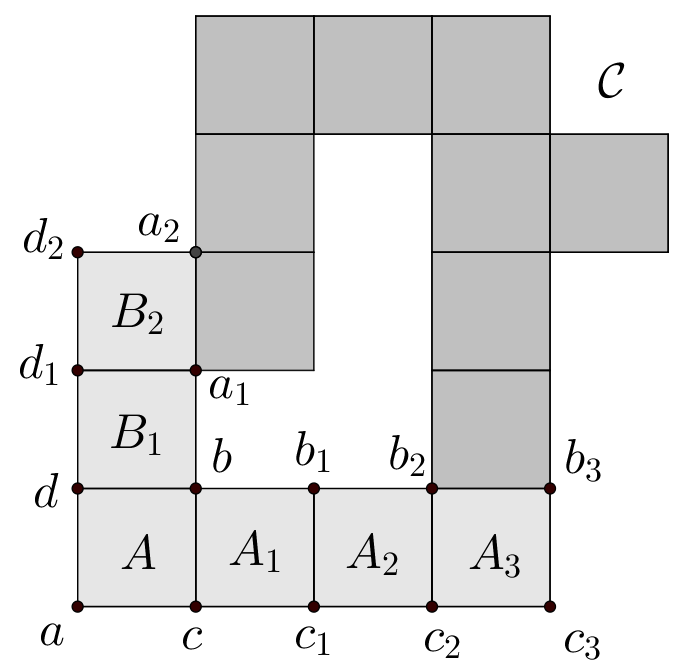}}
		\subfloat[]{\includegraphics[scale=0.6]{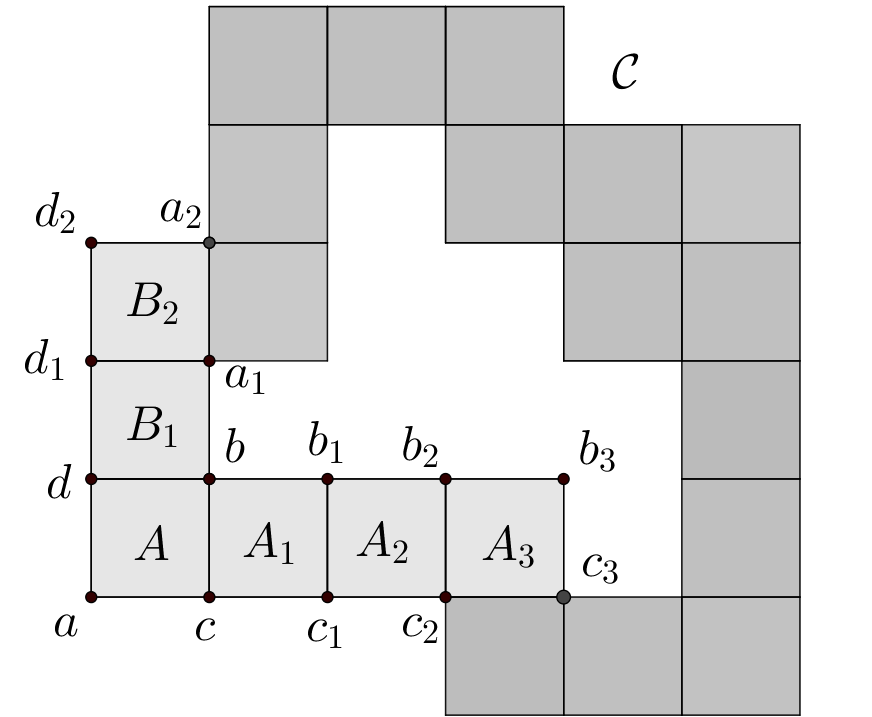}}
		\subfloat[]{\includegraphics[scale=0.6]{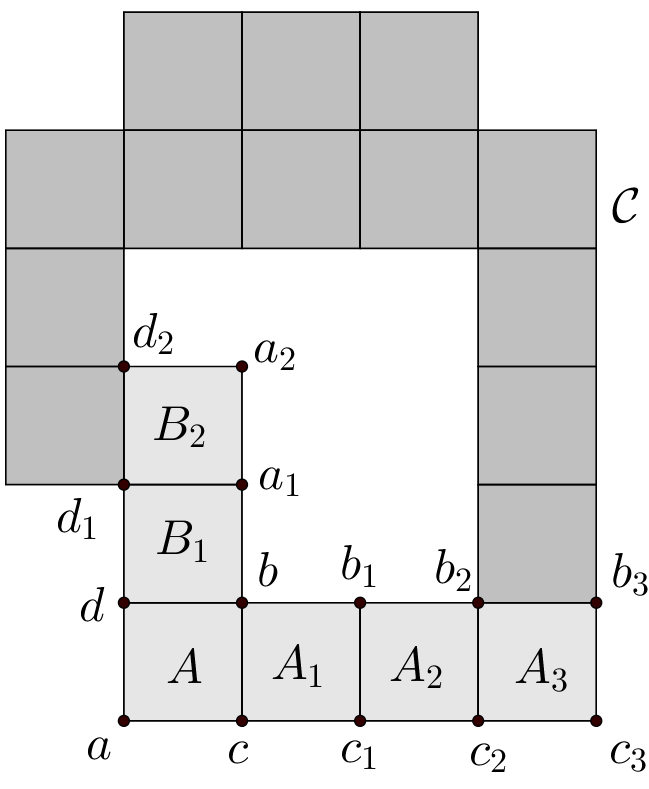}} 
		\subfloat[]{\includegraphics[scale=0.6]{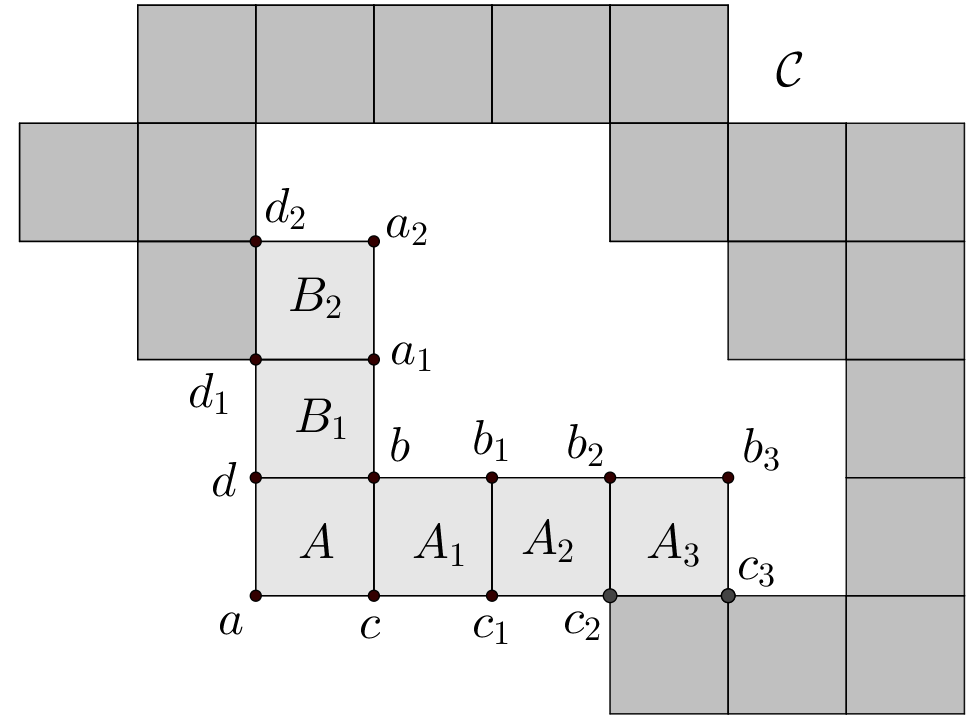}}
		\caption{Examples of the different cases of $(\cL,\cC)$-polyominoes}
		\label{Figura:esempio P(L,C)}
	\end{figure}
	If $\cP$ is an $(\mathcal{L},\mathcal{C})$-polyomino, the following related polyominoes will be used along the paper:
	\begin{itemize}
		\item $\cP_1=\cP\backslash [A,A_r]$; 
		\item $\cP_2=\cP\backslash [A,B_s]$; 
		\item $\cP_3=\cP\backslash ([A,A_r]\cup [A,B_s])=\mathcal{C}$;
		\item $\cP_4=\cP\backslash \{A,A_1,B_1\}$;
		\item $\cP'_1=\cP\backslash [A_1,A_r]$;
		\item $\cP'_2=\cP\backslash [B_1,B_s]$.  
	\end{itemize}
\end{definition}

\begin{lemma}
	Let $\cP$ be an $(\cL,\cC)$-polyomino. Then $S_\cP/(I_{\cP},x_a,x_d,x_c,x_b)\cong K[\cP_4]$.
	\label{isomorfismoP4}
\end{lemma}
\begin{proof}
	Observe that $I_\cP$ can be written in the following way 
	\begin{align*}
	&I_\cP=I_{\cP_4}+ (x_a x_b -x_b x_c)+\sum_{i=1}^r(x_a x_{a_i}-x_c x_{d_i})+\sum_{i=1}^r(x_d x_{a_i}-x_b x_{d_i})+\\
	&\sum_{i=1}^s(x_a x_{b_i}-x_d x_{c_i})+\sum_{i=1}^r(x_c x_{b_i}-x_b x_{c_i}).
	\end{align*}
	It follows that $(I_{\cP},x_a,x_d,x_c,x_b)=(I_{\cP_4},x_a,x_d,x_c,x_b)
	$, in particular $$S_\cP/(I_{\cP},x_a,x_d,x_c,x_b)=S_\cP/(I_{\cP_4},x_a,x_d,x_c,x_b)\cong S_{\cP_4}/I_{\cP_4}=K[\cP_4]$$.
\end{proof}

\begin{proposition}\label{Proposizione: K[P_i] è dominio, sapendo che I è primo}
	Let $\cP$ be an $(\cL,\cC)$-polyomino. If $I_{\cP}$ is prime, then $K[\cP_i]$ and $K[\cP'_j]$ are domains for $i\in [4]$ and $j\in\{1,2\}$. 
\end{proposition}
\begin{proof}
	We may assume that $\cP$ is an $(\cL,\cC)$-polyomino such that $V(\cL)\cap V(\cC)=\{a_{s-1},a_s, b_{r-1},b_r\}$, since similar arguments can be used in the other cases. We prove that $K[\cP_1]$ is a domain. Observe that $I_\cP$ is a toric ideal since $I_{\cP}$ is a prime binomial ideal. Then there exists a map $\phi: S_{\cP}\to K[t_1^{\pm 1},\dots,t_d^{\pm1}]$ with $x_{ij}\mapsto \mathbf{t}^{\mathbf{a_{ij}}} \footnote{If $\mathbf{t}=(t_1,\dots,t_d)$ and $\mathbf{a}=(a_1,\dots,a_d)\in \ZZ^d$, then $\mathbf{t}^{\mathbf{a}}=t_1^{a_1}\dots t_d^{a_d}$.}$ for all $(i,j)\in V(\cP)$ such that $I_{\cP}=\ker \phi$. Let $\mathcal{V}=\{a,c,c_1,\dots,c_r,b_1,\dots,b_{r-2}\}$, we define $\phi_\mathcal{V}:S_{\cP}\to K[t_1^{\pm 1},\dots,t_d^{\pm1}]$, by $\phi_\cV(x_v)=0$ if $v\in \cV$ and $\phi_\cV(x_v)=\phi(x_v)$ otherwise. Put $J_\cV:=(I_{\cP},\{x_v\mid v\in \cV\})$, we prove that $J_\cV=\ker \phi_\cV$. If $f\in \ker \phi_\cV$, we can write $f=\tilde{f}+\beta g$ where $\beta\in S_\cP$, $g\in (\{x_v\mid v\in \cV\})$ and $\tilde{f}$ not containing variables in the set $\{x_v\mid v\in \cV\}$. Since $\phi_\cV(f)=0$, we have $\phi(\tilde{f})=0$, so $\tilde{f}\in \ker\phi=I_{\cP}$. For the other inclusion it suffices to prove that $I_\cP \subseteq \ker\phi_\cV$. In such a case observe that, for this configuration, if $f=x_{i_1}x_{i_2}-x_{j_1}x_{j_2}$ is a generator of $I_\cP$ then $\{x_{i_1},x_{i_2}\}\cap \{x_v\mid v\in \cV\}\neq \emptyset$ if and only if $\{x_{j_1},x_{j_2}\}\cap \{x_v\mid v\in \cV\}\neq \emptyset$, so in all possible cases we have $\phi_\cV(f)=0$. Therefore $J_\cV=\ker\phi_\cV$ and $J_\cV$ is a prime ideal. As in Lemma \ref{isomorfismoP4}, we have also that $J_\cV=(I_{\cP_1},x_a,x_c,x_{c_i},x_{b_j}:i\in[r],j\in[r-2])$, and $K[\cP_1]\cong S_{\cP}/J$ is a domain. The proof for this case is done. For the other polyominoes the proof is analogue, considering: 
\begin{itemize}	
	\item for $\cP_2$ the set $\cV=\{a,d,d_1,\dots,d_s,a_1,\dots,a_{s-2}\}$;
	\item for $\cP_3$ the set $\cV=\{a,b,c,d,c_1,\ldots, c_r, d_1,\dots,d_s,a_1,\dots,a_{s-2},b_1,\ldots,b_{r-2}\}$; 
	\item for $\cP_4$ the set $\cV=\{a,b,c,d\}$;
	\item for $\cP_1'$ the set $\cV=\{b_1,\dots,b_{r-2},c_1,\dots,c_r\}$;
	\item for $\cP_2'$ the set $\cV=\{a_1,\dots,a_{s-2},d_1,\dots,d_s\}$.
	\end{itemize}
\end{proof}

\noindent If $\cP$ is an $(\cL,\cC)$-polyomino, our aim is to provide a formula for the Hilbert-Poincar\'e series of $K[\cP]$, involving the Hilbert-Poincar\'e series of $K[\cP_1]$, $K[\cP_2]$, $K[\cP_3]$ and $K[\cP_4]$ in the hypotheses that $K[\cP]$ is an integral domain. In particular, if $(i_1,i_2,i_3,i_4)$ is a permutation of the set $\{a,b,c,d\}$, our strategy consists in considering the following four short exact sequences:

\begin{footnotesize}
$$0 \longrightarrow S_\cP/(I_{\cP}:x_{i_1})  \longrightarrow S_\cP/I_{\cP}  \longrightarrow S_\cP/(I_{\cP},x_{i_1}) \longrightarrow0  $$
$$0 \longrightarrow S_\cP/((I_{\cP},x_{i_1}):x_{i_2})  \longrightarrow S_\cP/(I_{\cP},x_{i_1}) \longrightarrow S_\cP/(I_{\cP},x_{i_1},x_{i_2}) \longrightarrow0  $$
$$0 \longrightarrow S_\cP/((I_{\cP},x_{i_1},x_{i_2}):x_{i_3})  \longrightarrow S_\cP/(I_{\cP},x_{i_1},x_{i_2}) \longrightarrow S_\cP/(I_{\cP},x_{i_1},x_{i_2},x_{i_3}) \longrightarrow0  $$
$$0 \longrightarrow S_\cP/((I_{\cP},x_{i_1},x_{i_2},x_{i_3}):x_{i_4})  \longrightarrow S_\cP/(I_{\cP},x_{i_1},x_{i_2},x_{i_3}) \longrightarrow S_\cP/(I_{\cP},x_a,x_d,x_c,x_b) \longrightarrow0  $$
\end{footnotesize}

\noindent From the exact sequences above, we will obtain the Hilbert-Poincar\'e series of $S_\cP/I_\cP$ by a repeated application of Proposition~\ref{Proposizione: la prima che serve per calcolare HP} and considering in each case a suitable permutation $(i_1,i_2,i_3,i_4)$ of the set $\{a,b,c,d\}$ in order to compute the Hilbert-Poincar\'e series of the rings in the intermediate steps. To reach our aim we provide several preliminary lemmas, distinguishing the different possibilities for the set $V(\cL)\cap V(\mathcal{\cC})$.

\begin{lemma}\label{Lemma: column+prodotti tensoriali - primo caso}
	Let $\cP$ be an  $(\cL,\cC)$-polyomino such that $V(\cL)\cap V(\mathcal{\cC})=\{a_{s-1},a_s,b_{r-1},b_r\}$. Suppose that $I_{\cP}$ is prime. Then:
	\begin{enumerate}[(1)] 
		\item $S_\cP/((I_\cP,x_a):x_d)\cong K[\cP_1]\otimes_K K[x_{b_1},\dots,x_{b_{r-2}}]$;
		\item $S_\cP/((I_\cP,x_a,x_d):x_c)\cong K[\cP_2]\otimes_K K[x_{a_1},\dots,x_{a_{s-2}}]$;
		\item $S_\cP/((I_\cP,x_a,x_d,x_c):x_b)\cong K[\cP_3]\otimes_K K[x_b,x_{a_1},\dots,x_{a_{s-2}},x_{b_1},\dots,x_{b_{r-2}}]$.
	\end{enumerate} 
\end{lemma}
\begin{proof}
	(1) Firstly observe that $I_{\cP}$ can be written in the following way:
	\begin{align*}
		&I_{\cP}=I_{\cP_1}+(x_ax_b-x_cx_d)+\sum_{i=1}^{s}(x_ax_{a_i}-x_cx_{d_i})+\sum_{j=1}^{r}(x_ax_{b_j}-x_dx_{c_j})+\\
		&\sum_{j=1}^{r}(x_cx_{b_j}-x_bx_{c_j})+\sum_{k,l\in[r]\atop k<l}(x_{c_k}x_{b_l}-x_{c_l}x_{b_k})+\\
		&(\{x_{c_{r-1}}x_{v}-x_{c_r}x_{u}\vert [c_{r-1},v]\in\cI(\cP), u=v-(1,0)\}),
	\end{align*}
	Now we describe the ideal $(I_{\cP},x_a)$ in $S_\cP$:
	\begin{align*}
		&(I_{\cP},x_a)=(I_{\cP_1},x_a)+(x_cx_d)+\sum_{i=1}^{s}(x_cx_{d_i})+\sum_{j=1}^{r}(x_dx_{c_j})+\\
		&\sum_{j=1}^{r}(x_cx_{b_j}-x_bx_{c_j})+\sum_{k,l\in[r]\atop k<l}(x_{c_k}x_{b_l}-x_{c_l}x_{b_k})+\\
		&(\{x_{c_{r-1}}x_{v}-x_{c_r}x_{u}\vert [c_{r-1},v]\in\cI(\cP), u=v-(1,0)\}).
	\end{align*}

	\noindent We prove that $(I_{\cP},x_a):x_d=(I_{\cP_1},x_a)+(x_c)+\sum_{i=1}^{r}(x_{c_i})$. It follows trivially from the previous equality that $(I_{\cP},x_a):x_d\supseteq(I_{\cP_1},x_a)+(x_c)+\sum_{i=1}^{r}(x_{c_i})$. Let $f\in S_\cP$ such that $x_df\in (I_\cP,x_a)$. Then 
	\begin{align*}
		&x_df=g+\alpha x_a+\beta x_cx_d+\sum_{i=1}^{s}\gamma_i x_cx_{d_i}+\sum_{j=1}^{r}\delta_j x_dx_{c_j}+\sum_{i=j}^{r}\omega_{j}(x_cx_{b_j}-x_bx_{c_j})+\\
		&+\sum_{k,l\in[r]\atop k<l}\nu_{kl}(x_{c_k}x_{b_l}-x_{c_l}x_{b_k})+\sum_{ [c_{r-1},v]\in\cI(\cP)\atop u=v-(1,0)} \lambda_{v}(x_{c_{r-1}}x_{v}-x_{c_r}x_{u}),
	\end{align*}
	where $g\in \cI_{\cP_1}$, $\alpha,\beta,\gamma_i,\delta_j,\omega_{j},\nu_{k,l}\lambda_{v}\in S_\cP$ for all $i,k\in[s],j,l\in[r]$ and for all $v\in V(\cP)$ such that $[c_{r-1},v]\in\cI(\cP)$. As a consequence:
	\begin{align*}
		&x_d\biggl(f-\beta x_c-\sum_{j=1}^{r}\delta_j x_{c_j}\biggr)=g+\alpha x_a+\sum_{i=1}^{s}(\gamma_i x_{d_i})x_c+\sum_{i=j}^{r}(\omega_{j}x_{b_j})x_c-\sum_{i=j}^{r}(\omega_jx_b)x_{c_j}+\\
		&+\sum_{k,l\in[r]\atop k<l}(\nu_{kl}x_{b_l})x_{c_k}-\sum_{k,l\in[r]\atop k<l}(\nu_{kl}x_{b_k})x_{c_l}+\biggl(\sum_{ [c_{r-1},v]\in\cI(\cP)\atop u=v-(1,0)} \lambda_{v}x_{v}\biggr)x_{c_{r-1}}+\\
		&-\biggl(\sum_{ [c_{r-1},v]\in\cI(\cP)\atop u=v-(1,0)}\lambda_{v}x_{u}\biggr)x_{c_r}.
	\end{align*}
	Hence we obtain that $x_d\bigl(f-\beta x_c-\sum_{j=1}^{r}\delta_jx_{c_j}\bigr)\in I_{\cP_1}+(x_a)+(x_c)+\sum_{i=1}^{r}(x_{c_i})$. Since $K[\cP_1]$ is a domain (Proposition \ref{Proposizione: K[P_i] è dominio, sapendo che I è primo}) and $a,c,c_i\notin V(\cP_1)$ for all $i\in[r]$ then $I_{\cP_1}+(x_a)+(x_c)+\sum_{i=1}^{r}(x_{c_i})$ is a prime ideal in $S_\cP$. Since $x_d\notin I_{\cP_1}$, we have $f-\beta x_c-\sum_{j=1}^{r}\delta_jx_{c_j}\in I_{\cP_1}+(x_a)+(x_c)+\sum_{i=1}^{r}(x_{c_i})$, so $f\in I_{\cP_1}+(x_a)+(x_c)+\sum_{i=1}^{r}(x_{c_i})$, that is $(I_{\cP},x_a):x_d\subseteq(I_{\cP_1},x_a)+(x_c)+\sum_{i=1}^{r}(x_{c_i})$. In conclusion we have $(I_{\cP},x_a):x_d=(I_{\cP_1},x_a)+(x_c)+\sum_{i=1}^{r}(x_{c_i})$ and as a consequence
	$S_\cP/((I_\cP,x_a):x_d)=S_\cP/(I_{\cP_1}+(x_a,x_c,x_{c_1},\dots,x_{c_r}))\cong S_{\cP_1}/I_{\cP_1}\otimes_K K[x_v\mid v\in V(\cP\setminus \cP_1)]/(x_a,x_c,x_{c_1},\dots,x_{c_r}) = K[\cP_1]\otimes_K K[x_{b_1},\dots,x_{b_{r-2}}]$. The claim (1) is proved.\\
	(2) By similar computations as in the first part of (1) we can prove that $(I_\cP,x_a,x_d):x_c=(I_{\cP_2},x_a,x_d)+\sum_{i=1}^s(x_{d_i})$, so claim (2) follows by using similar arguments as in the last part in (1).\\
	(3) The argument is the same, considering that $(I_\cP,x_a,x_d,x_c):x_b=(I_{\cP_3},x_a,x_d,x_c)+\sum_{i=1}^s(x_{d_i})+\sum_{j=1}^r(x_{c_i})$ can be proved using computations similar to the previous cases.
\end{proof}

\noindent In the previous result we examine, for a $(\cL,\cC)$-polyomino, the case $V(\cL)\cap V(\mathcal{\cC})=\{a_{s-1},a_s,b_{r-1},b_r\}$. In order to examine the other cases we need other preliminary results involving the polyominoes $\cP'_1$ and $\cP'_2$.

\begin{lemma}
	Let $\cP$ be an $(\cL,\cC)$-polyomino. Then
	\begin{enumerate}
	 \item[(1)] $S_{\cP_2'} /(I_{\cP'_2},x_b,x_c)\cong K[\cP_2]$. Moreover, if $I_{\cP_2}$ is a prime ideal then $(I_{\cP'_2},x_b,x_c)$ is a prime ideal of $S_\cP$.
	 \item[(2)] $S_{\cP'_1} /(I_{\cP'_1},x_b,x_d)\cong K[\cP_1]$. Moreover, if $I_{\cP_1}$ is a prime ideal then $(I_{\cP'_1},x_b,x_d)$ is a prime ideal of $S_\cP$.
\end{enumerate}	
	\label{P'_2}
\end{lemma}
\begin{proof}
	(1) Let $\mathcal{R}$ be the polyomino obtained from the cells of $\cP_2$ and renaming the vertices $b$ and $c$ respectively by $d$ and $a$, in particular $S_\mathcal{R}=K[x_v \mid v\in V(\cP'_2)\setminus \{b,c\}]$. Observe that
	$$I_{\cP'_2}=I_\mathcal{R}+ (x_a x_b-x_c x_d)+\sum_{i=1}^r(x_c x_{b_i}-x_b x_{c_i})$$
	So $(I_{\cP'_2},x_b,x_c)=(I_\mathcal{R},x_b,x_c)$ and in particular $S_{\cP'_2}/(I_{\cP'_2},x_b,x_c)=S_{\cP'_2}/(I_\mathcal{R},x_b,x_c)\cong S_\mathcal{R}/I_\mathcal{R} = K[\mathcal{R}]\cong K[\cP_2]$, since $x_b,x_c$ do not belong to the support of any element of $I_\mathcal{R}$ and observing that, apart from the name of the vertices involved, $\mathcal{R}=\cP_2$. Furthermore $S_\cP/(I_{\cP'_2},x_b,x_c)\cong S_{\cP'_2}/(I_{\cP'_2},x_b,x_c)\otimes_K K[x_v\mid v\in V(\cP\setminus \cP'_2)]\cong K[\cP_2]\otimes_K K[x_v\mid v\in V(\cP\setminus \cP'_2)]$, so also the last claim follows. \\
	(2) The result can be obtained arguing as in the proof of (1). Indeed the arrangements involved in these situations can be considered the same up to one reflection and one rotation.
\end{proof}

\begin{lemma}\label{Lemma: column+prodotti tensoriali - secondo caso}
	Let $\cP$ be an  $(\cL,\cC)$-polyomino such that $V(\cL)\cap V(\mathcal{\cC})=\{d_{s-1},d_s,b_{r-1},b_r\}$. Suppose that $I_{\cP}$ is prime. Then:
	\begin{enumerate}[(1)] 
		\item $S_\cP/((I_\cP,x_c):x_b)\cong K[\cP_1]\otimes_K K[x_{b_1},\dots,x_{b_{r-2}}]$;
		\item $S_\cP/((I_\cP,x_b,x_c):x_a)\cong K[\cP_2]\otimes_K K[x_{d_1},\dots,x_{d_{s-2}}]$;
		\item $S_\cP/((I_\cP,x_a,x_b,x_c):x_d)\cong K[\cP_3]\otimes_K K[x_d,x_{b_1},\dots,x_{b_{r-2}},x_{d_1},\dots,x_{d_{s-2}}]$.
	\end{enumerate} 
\end{lemma}
\begin{proof}
	Arguing as in Lemma~\ref{Lemma: column+prodotti tensoriali - primo caso}, we obtain the equalities of the following ideals:
	\begin{enumerate}[(1)]
		\item $(I_\cP,x_c):x_b=(I_{\cP_1},x_c)+(x_a)+\sum_{i=1}^r(x_{c_i})$
		\item $(I_\cP,x_b,x_c):x_a=(I_{\cP'_2},x_b,x_c)+\sum_{i=1}^s(x_{a_i})$.
		\item $(I_\cP,x_a,x_b,x_c):x_d=(I_{\cP_3},x_a,x_b,x_c)+\sum_{i=1}^s(x_{a_i})+\sum_{i=1}^r(x_{c_i})$.
	\end{enumerate}
	In particular, the second equality above holds since $(I_{\cP'_2},x_b,x_c)$ is a prime ideal by Lemma~\ref{P'_2}. By the same lemma we have also $S_{\cP_2'} /(I_{\cP'_2},x_b,x_c)\cong K[\cP_2]$, from which claim (2) derives. For the sake of completeness we provide its proof.\\
	Observe that $I_{\cP}$ can be written in the following way:
	\begin{align*}
		&I_{\cP}=I_{\cP'_2}+\sum_{i=1}^{s}(x_ax_{a_i}-x_cx_{d_i})+\sum_{i=1}^{s}(x_dx_{a_i}-x_bx_{d_i})+\sum_{k,l\in[s]\atop k<l}(x_{d_k}x_{a_l}-x_{d_l}x_{a_k})+\\
		&+(\{x_{a_{s}}x_{v}-x_{a_{s-1}}x_{u}\vert [v, a_s]\in\cI(\cP), u=v+(0,1)\}),
	\end{align*}
	It follows:
	\begin{align*}
		&(I_{\cP},x_b,x_c)=(I_{\cP'_2},x_b,x_c)+\sum_{i=1}^{s}(x_ax_{a_i})+\sum_{i=1}^{s}(x_dx_{a_i})+\sum_{k,l\in[s]\atop k<l}(x_{d_k}x_{a_l}-x_{d_l}x_{a_k})\\
		&+(\{x_{a_{s}}x_{v}-x_{a_{s-1}}x_{u}\vert [v, a_s]\in\cI(\cP), u=v+(0,1)\}),
	\end{align*}
	We prove that $(I_{\cP},x_b,x_c):x_a=(I_{\cP'_2},x_b,x_c)+\sum_{i=1}^{s}(x_{a_i})$. From the previous equality it follows that $(I_{\cP},x_b,x_c):x_a\supseteq (I_{\cP'_2},x_b,x_c)+\sum_{i=1}^{s}(x_{a_i})$. Let $f\in S$ such that $x_af\in (I_\cP,x_b,x_c)$. Then 
	\begin{align*}
		&x_af=g+\sum_{i=1}^{s}\gamma_i x_ax_{a_i}+\sum_{j=1}^{s}\delta_j x_dx_{a_j}+\sum_{k,l\in[s]\atop k<l}\nu_{kl}(x_{d_k}x_{a_l}-x_{d_l}x_{a_k})+\\
		&+\sum_{ [v,a_s]\in\cI(\cP)\atop u=v+(0,1)} \lambda_{v}(x_{a_s} x_{v}-x_{a_{s-1}}x_{u}),
	\end{align*}
	where $g\in (\cI_{\cP'_2},x_b,x_c)$, $\gamma_i,\delta_j,\nu_{k,l}\lambda_{v}\in S_\cP$ for all $i,k,j,l\in[s]$ and for all $v\in V(\cP)$ such that $[v,a_s]\in\cI(\cP)$. As a consequence:
	\begin{align*}
		&x_a\biggl(f-\sum_{i=1}^{s}\gamma_i x_{a_i}\biggr)=g+\sum_{j=1}^{s}(\delta_j x_{d})x_{a_j}+\sum_{k,l\in[s]\atop k<l}(\nu_{kl}x_{d_k})x_{a_l}-\sum_{k,l\in[s]\atop k<l}(\nu_{kl}x_{d_l})x_{a_k}+\\
		&+\biggl(\sum_{ [v,a_s]\in\cI(\cP)\atop u=v+(0,1)} \lambda_{v}x_{v}\biggr)x_{a_s}-\biggl(\sum_{ [v,a_s]\in\cI(\cP)\atop u=v+(0,1)}\lambda_{v}x_{u}\biggr)x_{a_{s-1}}.
	\end{align*}
	Hence we obtain that $x_a\bigl(f-\sum_{i=1}^{s}\gamma_i x_{a_i}\bigr)\in (I_{\cP'_2},x_b,x_c)+\sum_{i=1}^{s}(x_{a_i})$. Since $(I_{\cP'_2},x_b,x_c)$ is prime and $a_i\notin V(\cP'_2)$ for all $i\in[s]$ then $(I_{\cP'_2},x_b,x_c)+\sum_{i=1}^{s}(x_{a_i})$ is a prime ideal in $S_\cP$. By being $x_a\notin I_{\cP'_2}$, we have $f-\sum_{i=1}^{s}\gamma_i x_{a_i}\in (I_{\cP'_2},x_b,x_c)+\sum_{i=1}^{s}(x_{a_i})$, so $f\in (I_{\cP'_2},x_b,x_c)+\sum_{i=1}^{s}(x_{a_i})$, that is $(I_{\cP},x_b,x_c):x_a\subseteq(I_{\cP'_2},x_b,x_c)+\sum_{i=1}^{s}(x_{a_i})$. In conclusion we have $(I_{\cP},x_b,x_c):x_a=(I_{\cP'_2},x_b,x_c)+\sum_{i=1}^{s}(x_{a_i})$ and as a consequence
	$S_\cP/((I_{\cP},x_b,x_c):x_a)=S_\cP/(I_{\cP'_2},x_b,x_c)+(x_{a_1},\dots,x_{a_s}))\cong S_{\cP_2}/(I_{\cP'_2},x_b,x_c)\otimes_K K[x_v\mid v\in V(\cP\setminus \cP_1)]/(x_{a_1},\dots,x_{a_s}) \cong K[\cP_2]\otimes_K K[x_{d_1},\dots,x_{d_{s-2}}].$
\end{proof}

\noindent We omit to provide the analogous result for the case $V(\cL)\cap V(\mathcal{\cC})=\{a_{s-1},a_s,c_{r-1},c_r\}$. In fact, we can reduce it to the case examined in the previous Lemma up to a rotation and a reflection.

\begin{lemma}\label{Lemma: column+prodotti tensoriali - quarto caso}
	Let $\cP$ be an  $(\cL,\cC)$-polyomino such that $V(\cL)\cap V(\mathcal{\cC})=\{d_{s-1},d_s,c_{r-1},c_r\}$. Suppose that $I_{\cP}$ is prime. Then:
	\begin{enumerate}[(1)] 
		\item $S_\cP/((I_\cP,x_b):x_c)\cong K[\cP_1]\otimes_K K[x_{c_1},\dots,x_{c_{r-2}}]$;
		\item $S_\cP/((I_\cP,x_b,x_c):x_d)\cong K[\cP_2]\otimes_K K[x_{d_1},\dots,x_{d_{s-2}}]$;
		\item $S_\cP/((I_\cP,x_b,x_c,x_d):x_a)\cong K[\cP_3]\otimes_K K[x_d,x_{c_1},\dots,x_{c_{r-2}},x_{d_1},\dots,x_{d_{s-2}}]$.
	\end{enumerate} 
\end{lemma}
\begin{proof}
	The claims follow reasoning as in Lemma~\ref{Lemma: column+prodotti tensoriali - primo caso}, obtaining the equalities of the following ideals:
	\begin{enumerate}[(1)]
		\item $(I_\cP,x_b):x_c=(I_{\cP'_1},x_b,x_d)+\sum_{i=1}^r(x_{b_i})$
		\item $(I_\cP,x_b,x_c):x_d=(I_{\cP'_2},x_b,x_c)+\sum_{i=1}^s(x_{a_i})$.
		\item $(I_\cP,x_b,x_c,x_d):x_a=(I_{\cP_3},x_b,x_c,x_d)+\sum_{i=1}^s(x_{a_i})+\sum_{i=1}^r(x_{b_i})$.
	\end{enumerate}
	In particular, the first equality follows from  the primality of $(I_{\cP'_1},x_b,x_d)$, the second equality follows from the primality of $(I_{\cP'_2},x_b,x_c)$, both by Lemma~\ref{P'_2}.
\end{proof}

\begin{theorem}\label{Teorema: serie di Hilbert - primo caso}
	Let $\cP$ be an $(\cL,\cC)$-polyomino. Suppose that $I_{\cP}$ is prime. Then:
	$$\mathrm{HP}_{K[\cP]}(t)=\frac{1}{1-t}\rHP_{K[\cP_4]}(t)+\frac{t}{1-t}\bigg[\frac{\rHP_{K[\cP_1]}(t)}{(1-t)^{r-2}}+\frac{\rHP_{K[\cP_2]}(t)}{(1-t)^{s-2}}+\frac{\rHP_{K[\cP_3]}(t)}{(1-t)^{s+r-3}} \bigg]$$
\end{theorem}

\begin{proof}
	Assume that $V(\cL)\cap V(\mathcal{\cC})=\{a_{s-1},a_s,b_{r-1},b_r\}$. Consider the following four short exact sequences:
	
	\begin{footnotesize}
		$$0 \longrightarrow S_\cP/(I_{\cP}:x_a)  \longrightarrow S_\cP/I_{\cP}  \longrightarrow S_\cP/(I_{\cP},x_a) \longrightarrow0  $$
	$$0 \longrightarrow S_\cP/((I_{\cP},x_a):x_d)  \longrightarrow S_\cP/(I_{\cP},x_a) \longrightarrow S_\cP/(I_{\cP},x_a,x_d) \longrightarrow0  $$
	$$0 \longrightarrow S_\cP/((I_{\cP},x_a,x_d):x_c)  \longrightarrow S_\cP/(I_{\cP},x_a,x_d) \longrightarrow S_\cP/(I_{\cP},x_a,x_d,x_c) \longrightarrow0  $$
	$$0 \longrightarrow S_\cP/((I_{\cP},x_a,x_d,x_c):x_b)  \longrightarrow S_\cP/(I_{\cP},x_a,x_d,x_c) \longrightarrow S_\cP/(I_{\cP},x_a,x_d,x_c,x_b) \longrightarrow0  $$
	\end{footnotesize}

	Since $I_{\cP}:x_a=I_{\cP}$, because $I_{\cP}$ is prime, the claim easily follows by repeated applications of Proposition \ref{Proposizione: la prima che serve per calcolare HP} and from Proposition~\ref{Hilber-tensoriale}, Lemma~\ref{isomorfismoP4} and Lemma \ref{Lemma: column+prodotti tensoriali - primo caso}.\\
	If $V(\cL)\cap V(\mathcal{\cC})=\{d_{s-1},d_s,b_{r-1},b_r\}$ the formula is obtained referring to Lemma~\ref{Lemma: column+prodotti tensoriali - secondo caso} by a suitable permutation of the set $\{a,b,c,d\}$. For symmetry, we obtain the claim also for the case $V(\cL)\cap V(\mathcal{\cC})=\{a_{s-1},a_s,c_{r-1},c_r\}$.\\	
	Finally, if $V(\cL)\cap V(\mathcal{\cC})=\{d_{s-1},d_s,c_{r-1},c_r\}$ we use again the same argument together with Lemma~\ref{Lemma: column+prodotti tensoriali - quarto caso}.
\end{proof}

\begin{coro}
	Let $\cP$ be an $(\cL,\cC)$-polyomino, suppose that $I_{\cP}$ is a prime ideal and $\cC$ is a simple polyomino. Then:
	$$\mathrm{HP}_{K[\cP]}(t)=\frac{h_{K[\cP_4]}(t)+t\big[h_{K[\cP_1]}(t)+h_{K[\cP_2]}(t)+(1-t)h_{K[\cP_3]}(t)\big]}{(1-t)^{\vert V(\cP)\vert -\mathrm{rank}\,\cP}}$$
	In particular $K[\cP]$ has Krull dimension $\vert V(\cP)\vert -\mathrm{rank}\,\cP$. 
	\label{HilbertCsimple}
\end{coro}
\begin{proof}
	Since $\cC$ is a simple polyomino, then $\cP_1$, $\cP_2$, $\cP_3$ and $\cP_4$ are simple polyominoes, so we have that $K[\cP_j]$ is a normal Cohen-Macaulay domain of dimension $\vert V(\cP_j)\vert -\mathrm{rank}\,\cP_j$ for $j\in\{1,2,3,4\}$ from \cite[Corollary 3.3]{def balanced} and \cite[Theorem 2.1]{Simple equivalent balanced}. We put $\vert V(\cP)\vert =n$ and $\mathrm{rank}\,\cP=p$. Observe that 
	\begin{itemize}
		\item $\vert V(\cP_1)\vert =n-2r$ and $\mathrm{rank}\,\cP_1=p-r-1$, so $\vert V(\cP_1)\vert -\mathrm{rank}\,\cP_1=n-p-r+1$;
		\item $\vert V(\cP_2)\vert =n-2s$ and $\mathrm{rank}\,\cP_2=p-s-1$, so $\vert V(\cP_2)\vert -\mathrm{rank}\,\cP_2=n-p-s+1$;
		\item $\vert V(\cP_3)\vert =n-2s-2r$ and $\mathrm{rank}\,\cP_3=p-r-s-1$, so $\vert V(\cP_3)\vert -\mathrm{rank}\,\cP_3=n-p-s-r+1$;
		\item $\vert V(\cP_4)\vert =n-4$ and $\mathrm{rank}\,\cP_4=p-3$, so $\vert V(\cP_4)\vert -\mathrm{rank}\,\cP_4=n-p-1$.
	\end{itemize}
	Then $n-p=\vert V(\cP_4)\vert -\mathrm{rank}\,\cP_4+1=\vert V(\cP_1)\vert -\mathrm{rank}\,\cP_1+(r-2)+1=\vert V(\cP_2)\vert -\mathrm{rank}\,\cP_2+(s-2)+1$ and $\vert V(\cP_3)\vert -\mathrm{rank}\,\cP_3+(s+r-3)+1=n-p-1$. Therefore the formula for $\mathrm{HP}_{K[\cP]}(t)$ in the statement follows from Theorem \ref{Teorema: serie di Hilbert - primo caso} after an easy computation. Finally, let $h(t)$ be the polynomial in the numerator of the formula. By \cite[Corollary 4.1.10]{Bruns_Herzog}, observe that $h(1)=h_{K[\cP_4]}(1)+h_{K[\cP_1]}(1)+h_{K[\cP_2]}(1)>0$, so $(1-t)$ does not divide $h(t)$, hence $K[\cP]$ has Krull dimension $\vert V(\cP)\vert -\mathrm{rank}\,\cP$.
\end{proof}

\section{Hilbert-Poincar\'e series of prime closed path polyominoes having no $L$-configuration} \label{Section: H-P series of prime closed path polyominoes having no L-configuration}

\noindent In this section we suppose that $\cP$ is a prime closed path polyomino having no $L$-configurations, so $\cP$ contains a ladder of at least three steps (\cite[Section 6]{Cisto_Navarra_closed_path}). Let $\mathcal{B}_1$, $\mathcal{B}_2$ and $\mathcal{B}_3$ be three maximal horizontal blocks of a ladder of $n$ steps in $\cP$, $n\geq 3$. Without loss of generality, we can assume  that there does not exist a maximal block $\mathcal{K}\neq\mathcal{B}_2,\mathcal{B}_3$ of $\cP$ such that $\{\mathcal{K},\mathcal{B}_1,\mathcal{B}_2\}$ is a ladder of three steps. Moreover, applying suitable reflections or rotations of $\cP$, we can suppose that the orientation of the ladder is right/up, as in Figure~\ref{ladder_vuota}.

\begin{figure}[h!]
	\centering
\includegraphics[scale=0.75]{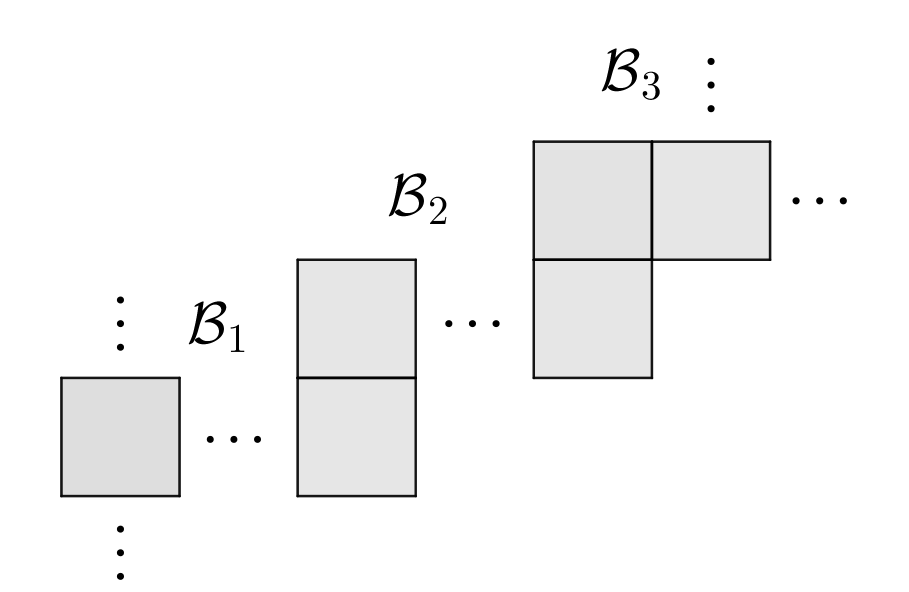}
\caption{}
\label{ladder_vuota}
\end{figure}

\noindent Our aim is to study the Hilbert-Poincar\'e series of the coordinate ring of $\cP$. We split our arguments in two cases. In the first case we suppose that at least one block between $\mathcal{B}_1$ or $\mathcal{B}_2$ contains exactly two cells, in the second one assume that $\mathcal{B}_1$ and $\mathcal{B}_2$ contain at least three cells.

\subsection{At least one block between $\mathcal{B}_1$ or $\mathcal{B}_2$ contains exactly two cells.}\label{ladder2}

We start with some preliminary definitions that we adopt throughout this subsection. Let $\mathcal{W}$ be a collection of cells consisting of an horizontal block $[A_s,A_1]$ of rank at least two, containing the cells $A_s,A_{s-1},\ldots,A_1$, a vertical block $[B_1,B_r]$ of rank at least two, containing the cells $B_1,B_2,\ldots,B_r$, and a cell $A$ not belonging to $[A_s,A_1]\cup[B_1,B_r]$, such that $V([A_s,A_1])\cap V([B_1,B_r])=\{b\}$, where $b$ is the lower right corner of $A$. Moreover we denote the left upper corner of $A$ by $a$, the lower right corner of $B_1$ by $d$, the lower right corner of $A_1$ by $c$. Moreover, let $b_i$ and $c_i$ be respectively the left upper and lower corners of $A_i$ for $i\in[s]$, let $a_j$ and $d_j$ be respectively the left and the right upper corners of $B_j$ for $j\in[r]$ (Figure~\ref{Figura:pentomino}).

\begin{figure}[h!]
	\centering
\includegraphics[scale=0.9]{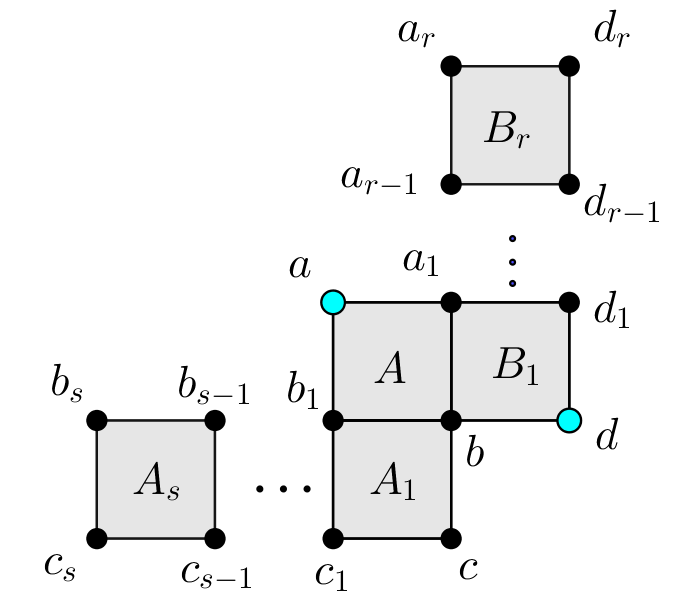}
\caption{A collection of cells $\mathcal{W}$}
\label{Figura:pentomino}
\end{figure}

\noindent Since $\cP$ has no $L$-configurations, it is trivial to check that $\cP$ contains a collection of cells $\mathcal{W}$ such that $[A_s,A_1]$ and $[B_1,B_r]$ are maximal blocks of $\cP$. In particular, if $\cM$ is the collection of cells such that $\cP=\cW\sqcup \cM$, then we call $\cW$:
		\begin{itemize}
		\item \textit{1-Configuration}, if $V(\cW)\cap V(\mathcal{\cM})=\{c_{s-1},c_s,d_{r-1},d_r\}$;
		\item \textit{2-Configuration}, if $V(\cW)\cap V(\mathcal{\cM})=\{b_{s-1},b_s,d_{r-1},d_r\}$.
		\end{itemize}
	\noindent Observe that just one of the following cases can occur:
		\begin{enumerate}[(1)]
		\item $\vert \mathcal{B}_1\vert =\vert \mathcal{B}_2\vert =2$. In such a case $s=2$ and $r=2$, $\mathcal{B}_1=[A_2,A_1]$ and $\mathcal{B}_2=[A,B_1]$, so we obtain an 1-Configuration.
		\item $\vert \mathcal{B}_1\vert > 2$ and $\vert \mathcal{B}_2\vert =2$. In such a case $s> 2$ and $r=2$, $\mathcal{B}_1=[A_s,A_1]$ and $\mathcal{B}_2=[A,B_1]$, so we have an 1-Configuration or a 2-Configuration depending on $\mathcal{M}\cap \{A_s\}$.
		\item  $\vert \mathcal{B}_1\vert = 2$ and $\vert \mathcal{B}_2\vert > 2$. In such a case, after a suitable rotation and reflection, consider a new ladder where $\mathcal{B}_1=[A_1,A]$ and $\mathcal{B}_2=[B_1,B_r]$, $s\geq 2$ and $r> 2$. Let $C$ be a cell of $\cP$ such that $I:=[C,A_1]$ is a maximal block of $\cP$. Therefore we obtain an 1-Configuration or a 2-Configuration depending on the position of the cell of $\cP\backslash I$ adjacent to $C$. 
		\end{enumerate}
		 The following related polyominoes will be essential in this subsection:
		\begin{itemize}
		\item $\mathcal{Q}=\cP \setminus\{A\}$;
		\item $\mathcal{Q}_1=\cP \setminus \{A,A_1,B_1\}$;
		\item $\cR_1=\cQ\setminus \{B_1\}$;
		\item $\cR_2=\cQ\setminus \{B_1,\ldots,B_s\}$;
		\item $\cF_1=\cQ\setminus \{A_1,\ldots,A_s\}$;
		\item $\cF_2=\cQ\setminus \{A_1,B_1,\ldots,B_s\}$.
		\end{itemize}		

\noindent Let $<^1$ be the total order on $V(\cP)$ defined as $u<^1 v$ if and only if, for $u = (i,j)$ and $v = (k,l)$, $i < k$, or $i = k$ and $j < l$. Let $Y\subset V(\cP)$ and let $<^Y_{\mathrm{lex}}$ be the lexicographic order in $S_\cP$ induced by the following order on the variables of $S_\cP$:
	\[ \mbox{for}\ u,v \in V(\cP)\qquad
	x_u<^Y_{\mathrm{lex}} x_v \Leftrightarrow
	\left\{
	\begin{array}{l}
	u\notin Y\ \mbox{and}\ v\in Y \\
	u,v\notin Y\ \mbox{and}\ u<^1 v \\
	u,v\in Y\ \mbox{and}\ u<^1 v
	\end{array}
	\right.
	\]
Considering Figure~\ref{Figura:pentomino}, from \cite[Theorem 4.9] {Cisto_Navarra_CM_closed_path} we know that there exists a set $L\subset V(\cP)$, with $a,d \in L$ and $b,c,a_1,b_1,c_1,d_1 \notin L$, such that the set of generators of $I_{\cP}$ forms the reduced Gr\"obner basis of $I_\cP$ with respect to $<^L_{\mathrm{lex}}$. Furthermore, in the case of 1-Configuration also $d_2,\ldots d_r\notin L$. For convenience we denote such a monomial order by $\prec_\cP$. Moreover, let $\prec_{\cQ}$, $\prec_{\cQ_1}$, $\prec_{\cR_1}$, $\prec_{\cR_2}$ be the monomial orders induced from $\prec_\cP$ respectively on the rings $S_Q$, $S_{\cQ_1}$, $S_{\cR_1}$, $S_{\cR_2}$. The following proposition will be useful. 
\begin{proposition}
Let $\cP$ be a closed path polyomino containing a collection of cells of type $\cW$. Then the set of the inner 2-minors of $\cQ$ is the reduced Gr\"obner basis of $I_\cQ$ with respect to the monomial order $\prec_{\cQ}$. The same holds for the polyominoes $\cQ_1$, $\cR_1$ and $\cR_2$ considering respectively the monomial orders $\prec_{\cQ_1}$, $\prec_{\cR_1}$ and $\prec_{\cR_2}$.
\label{orders}
\end{proposition}
\begin{proof}
 Let $f,g$ be two generators of $I_\cQ\subset I_\cP$. Since every $S$-polynomial $S(f,g)$ reduces to zero in $I_\cP$ then the conditions in lemmas in \cite[Section 3] {Cisto_Navarra_CM_closed_path} are satisfied for the collection of cells $\cP$. Apart from the occurrences $f=x_b x_{c_1}-x_c x_{b_1}$ and $g=x_b x_{d_1} -x_d x_{a_1}$, in which the leading terms of $f$ and $g$ have the greatest common divisor equal to 1, the other conditions of the mentioned lemmas do not involve the cell $A$. So the same conditions hold also for the collection of cells $Q$, hence $S(f,g)$ reduces to zero also in $I_\cQ$. By the same argument also the second claim in the statement holds.
  \end{proof}  

\begin{remark} \rm
		Observe that $\cQ_1$, $\cR_1$, $\cR_2$, $\cF_1$ and $\cF_2$ are simple polyominoes, so their related coordinate rings are normal Cohen-Macaulay domains whose Krull dimension is given by the difference between the number of vertices and the number of cells of the fixed polyomino (see \cite[Corollary 3.3]{def balanced} and \cite[Theorem 2.1]{Simple equivalent balanced}). The polyomino $\cQ$ is not simple but it is a weakly closed path and it is easy to see that $\cQ$ has a weak ladder in the cases which we are studying. Therefore $I_\cQ$ is a prime ideal (equivalently $K[\cQ]$ is a domain) from \cite[Proposition 4.5]{Cisto_Navarra_weakly}. Moreover, from Proposition~\ref{orders} and arguing as in the proof of \cite[Theorem 4.10]{Cisto_Navarra_CM_closed_path} we also obtain that $K[\cQ]$ is a normal Cohen-Macaulay domain. 
\label{RemarkQ}
\end{remark}

\noindent We are going to use all these introductory facts in the proofs of the next results. With abuse of notation we refer to $\mathrm{in}(I_\cP)$, $\mathrm{in}(I_\cQ)$, $\mathrm{in}(I_{\cQ_1})$, $\mathrm{in}(I_{\cR_1})$, $\mathrm{in}(I_{\cR_2})$ respectively for the initial ideals of $I_\cP$ with respect to $\prec_\cP$, of $I_\cQ$ with respect to $\prec_\cQ$, of $I_{\cQ_1}$ with respect to $\prec_{\cQ_1}$, of $I_{\cR_1}$ with respect to $\prec_{\cR_1}$ and of $I_{\cR_2}$ with respect to $\prec_{\cR_2}$.

\begin{proposition}
Let $\cP$ be a closed path polyomino containing a collection of cells of type $\cW$. Then 
$$\mathrm{HP}_{K[\cP]}(t)=\mathrm{HP}_{K[\cQ]}(t)+\frac{t}{1-t}\mathrm{HP}_{K[\cQ_1]}(t)$$ 
\label{HilbertQ}
\end{proposition}
\begin{proof}
Observe that:
\begin{align*}
&I_\cP= I_{\cQ}+ (x_{b_1}x_{a_1}-x_a x_b)+(x_{b_1} x_{d_1}-x_a x_d)+(x_{c_1} x_{a_1}-x_a x_c),\\
&I_\cP= I_{\cQ_1}+ (x_{b_1}x_{a_1}-x_a x_b)+(x_{b_1} x_{d_1}-x_a x_d)+(x_{c_1} x_{a_1}-x_a x_c)+\\
&+\sum_{i=1}^r (x_b x_{d_i}-x_d x_{a_i})+\sum_{i=1}^s (x_b x_{c_i}-x_c x_{b_i}).  
\end{align*}
From Proposition~\ref{orders} we obtain:
\begin{align*}
& \mathrm{in}(I_\cP)=\mathrm{in}(I_\cQ)+(x_a x_b)+ (x_a x_d) + (x_a x_c),\\
& \mathrm{in}(I_\cP)=\mathrm{in}(I_{\cQ_1})+(x_a x_b)+ (x_a x_d) + (x_a x_c) + (\{\max_{\prec_\cP}\{x_b x_{d_i},x_d x_{a_i}\}:i\in[r]\})+\\
&+ (\{\max_{\prec_\cP} \{x_b x_{c_i},x_c x_{b_i}:i\in [s]\}).
\end{align*}

From the above equalities it is not difficult to see that:
\begin{itemize}
\item $(\mathrm{in}(I_\cP),x_a)=(\mathrm{in}(I_\cQ),x_a)$, in particular $S_\cP/(\mathrm{in}(I_\cP),x_a)= S_\cP/(\mathrm{in}(I_\cQ),x_a)\cong S_\cQ/\mathrm{in}(I_\cQ)$.
\item $\mathrm{in}(I_\cP):x_a=(\mathrm{in}(I_{\cQ_1}),x_b,x_c,x_d)$ (see for instance \cite[Proposition 1.2.2]{monomial_ideals}), in particular $S_\cP/(\mathrm{in}(I_\cP):x_a)=S_\cP/(\mathrm{in}(I_{\cQ_1}),x_b,x_c,x_d)\cong S_{\cQ_1}/\mathrm{in}(I_{\cQ_1})\otimes_K K[x_a]$. 
\end{itemize}
Consider the following exact sequence:
$$0 \longrightarrow S_\cP/(\mathrm{in}(I_\cP):x_a) \longrightarrow S_\cP/\mathrm{in}(I_\cP) \longrightarrow S_\cP/(\mathrm{in}(I_\cP),x_a) \longrightarrow0  $$

\noindent Since for every graded ideal $I$ of a standard graded $K$-algebra $S$ and for every monomial order $<$ on $S$ it is verified that $S/I$ and $S/\mathrm{in}_<(I)$ have the same Hilbert function (see \cite[Corollary 6.1.5]{monomial_ideals}), then from the above computations and from Propositions~\ref{Proposizione: la prima che serve per calcolare HP} and \ref{Hilber-tensoriale} we obtain $\mathrm{HP}_{K[\cP]}(t)=\mathrm{HP}_{K[\cQ]}(t)+\frac{t}{1-t}\mathrm{HP}_{K[\cQ_1]}(t)$.
\end{proof}

\noindent We observed that $\cQ$ is not a simple polyomino. Our aim is to provide a formula for the Hilbert-Poincar\'e series of $K[\cP]$ involving the Hilbert-Poincar\'e series related to the coordinate rings of simple polyominoes. By the previous result, since $\cQ_1$ is a simple polyomino, we have to study the  Hilbert-Poincar\'e series of $K[\cQ]$.  We examine 1-Configuration and 2-Configuration separately.

\begin{theorem}
Let $\cP$ be a closed path polyomino containing a collection of cells of type $\cW$ with the occurrence of 1-Configuration. Then
$$\mathrm{HP}_{K[\cP]}(t)=\frac{h_{K[\cR_1]}(t)+t\big[h_{K[\cR_2]}(t)+h_{K[\cQ_1]}(t)\big]}{(1-t)^{\vert V(\cP)\vert -\mathrm{rank}\,\cP}}$$
In particular, the Krull dimension of $K[\cP]$  is $\vert V(\cP)\vert -\mathrm{rank}\,\cP$.
\label{Hilbert-1conf}
\end{theorem}
\begin{proof}
Observe that:
\begin{align*}
& I_\cQ= I_{\cR_1}+\sum_{i=1}^r (x_b x_{d_i}-x_d x_{a_i}),\\
& I_\cQ= I_{\cR_2}+\sum_{i=1}^r (x_b x_{d_i}-x_d x_{a_i})+ \sum_{k,l\in[r]\atop k<l}(x_{a_k}x_{d_l}-x_{a_l}x_{d_k})+ \\
&+ (\{x_{a_{r-1}}x_{v}-x_{a_r}x_{u}\vert [a_{r-1},v]\in\cI(\cQ), u=v-(0,1)\}).
\end{align*}
From Proposition~\ref{orders} we obtain:
\begin{align*}
& \mathrm{in}(I_\cQ)=\mathrm{in}(I_{\cR_1})+ \sum_{i=1}^r (x_d x_{a_i}), \\
& \mathrm{in}(I_\cQ)=\mathrm{in}(I_{\cR_2})+ \sum_{i=1}^r (x_d x_{a_i})+ (\{\max_{\prec_\cQ} \{x_{a_k}x_{d_l},x_{a_l}x_{d_k}\}\vert k,l\in[r], k<l\})+\\
& +(\{\max_{\prec_\cQ} \{x_{a_{r-1}}x_{v},x_{a_r}x_{u}\}\vert [a_{r-1},v]\in\cI(\cQ), u=v-(0,1)\}).
\end{align*}
From the above equalities is not difficult to see that:
\begin{itemize}
\item $(\mathrm{in}(I_\cQ),x_d)=(\mathrm{in}(I_{\cR_1}),x_d)$, in particular $S_\cQ/(\mathrm{in}(I_\cQ),x_d)= S_\cQ/(\mathrm{in}(I_{\cR_1}),x_d)\cong S_{\cR_1}/\mathrm{in}(I_{\cR_1})$.
\item $\mathrm{in}(I_\cQ):x_d=\mathrm{in}(I_{\cR_2})+\sum_{i=1}^r(x_{a_i})$, in particular $S_\cQ/(\mathrm{in}(I_\cQ):x_d)\cong S_{\cR_2}/\mathrm{in}(I_{\cR_2})\otimes_K K[x_{d_1},\ldots,x_{d_{r-2}}]$. 
\end{itemize}
So, arguing as in the proof of Proposition~\ref{HilbertQ}, we obtain $\mathrm{HP}_{K[\cQ]}(t)=\mathrm{HP}_{K[\cR_1]}(t)+t\cdot \frac{\mathrm{HP}_{K[\cR_2]}(t)}{(1-t)^{r-2}}$. Combining such an equality with the claim of Proposition~\ref{HilbertQ} we have: 
$$ \mathrm{HP}_{K[\cP]}(t)=\mathrm{HP}_{K[\cR_1]}(t)+t\cdot \left(\frac{\mathrm{HP}_{K[\cR_2]}(t)}{(1-t)^{r-2}}+\frac{\mathrm{HP}_{K[\cQ_1]}(t)}{1-t}\right)$$
 Set $\vert V(\cP)\vert =n$ and $\mathrm{rank}\,\cP=p$. Observe that 
	\begin{itemize}
		\item $\vert V(\cR_1)\vert =n-2$ and $\mathrm{rank}\,\cR_1=p-2$, so $\vert V(\cR_1)\vert -\mathrm{rank}\,\cR_1=n-p$ and this is the Krull dimension of $K[\cR_1]$ since $\cR_1$ is simple;
		\item $\vert V(\cR_2)\vert =n-2r+1$ and $\mathrm{rank}\,\cP_2=p-r-1$, so $\vert V(\cR_2)\vert -\mathrm{rank}\,\cR_2=n-p-r+2$ and this is the Krull dimension of $K[\cR_2]$;
		\item $\vert V(\cQ_1)\vert =n-4$ and $\mathrm{rank}\,\cQ_1=p-3$, so $\vert V(\cQ_1)\vert -\mathrm{rank}\,\cQ_1=n-p-1$ and this is the Krull dimension of $K[\cQ_1]$.
	\end{itemize}
Therefore, by easy computations, we obtain the formula for $\mathrm{HP}_{K[\cP]}(t)$ in the statement. Finally, because of the Cohen-Macaulay property of $K[\cR_1]$, $K[\cR_2]$ and $K[\cQ_1]$ and by \cite[Corollary 4.1.10]{Bruns_Herzog}, we have that $h_{K[\cR_1]}(1)+h_{K[\cR_2]}(1)+h_{K[\cQ_1]}(1)>0$, so $\dim K[\cP]=\vert V(\cP)\vert -\mathrm{rank}\,\cP$.
\end{proof}

\noindent Now we want to study the 2-Configuration. In such a case we do not need to use the initial ideals.

\begin{theorem}
Let $\cP$ be a closed path polyomino containing a collection of cells of type $\cW$ with the occurence of 2-Configuration. Then 
$$\mathrm{HP}_{K[\cP]}(t)=\frac{(1+t)h_{K[\cQ_1]}(t)+t\big[h_{K[\cF_1]}(t)+h_{K[\cF_2]}(t)\big]}{(1-t)^{\vert V(\cP)\vert -\mathrm{rank}\,\cP}}$$
In particular, the Krull dimension of $K[\cP]$  is $\vert V(\cP)\vert -\mathrm{rank}\,\cP$.
\label{Hilbert-2conf}
\end{theorem}
\begin{proof}
Arguing as in Lemma~\ref{Lemma: column+prodotti tensoriali - primo caso} we obtain the following equalities:
\begin{enumerate}[(1)]
\item $I_\cQ : x_c=I_\cQ$;
\item $(I_\cQ,x_c):x_b= I_{\cF_1}+(x_c)+\sum_{i=1}^s(x_{c_i})$;
\item $(I_\cQ,x_b,x_c):x_d= I_{\cF_2}+(x_b,x_c)+\sum_{i=1}^r(x_{a_i})$.
\item $(I_\cQ,x_b,x_c,x_d)=(I_{\cQ_1},x_b,x_c,x_d)$
\end{enumerate}
Again by the same arguments of Lemma~\ref{Lemma: column+prodotti tensoriali - primo caso} we obtain also the following:
\begin{enumerate}[(1)]
\item $S_\cQ/(I_\cQ : x_c)=K[\cQ]$;
\item $S_\cQ/((I_\cQ,x_c):x_b)\cong K[\cF_1]\otimes_K K[x_{b_1},\ldots,x_{b_{s-2}}]$;
\item $S_\cQ/((I_\cQ,x_b,x_c):x_d)\cong K[\cF_2]\otimes_K K[x_d,x_{d_1},\ldots,x_{d_{r-2}}]$;
\item $S_\cQ/(I_\cQ,x_b,x_c,x_d)\cong K[\cQ_1]$
\end{enumerate}
Now considering the suitable exact sequences and arguing as in Theorem~\ref{Teorema: serie di Hilbert - primo caso}, the following holds:
$$\mathrm{HP}_{K[\cQ]}(t)=\frac{1}{1-t}\mathrm{HP}_{K[\cQ_1]}+\frac{t}{1-t}\left[\frac{\mathrm{HP}_{K[\cF_1]}(t)}{(1-t)^{s-2}}+\frac{\mathrm{HP}_{K[\cF_2]}(t)}{(1-t)^{r-1}}\right]$$
So, from Theorem~\ref{HilbertQ} we have:
$$\mathrm{HP}_{K[\cP]}(t)=\frac{1+t}{1-t}\mathrm{HP}_{K[\cQ_1]}+\frac{t}{1-t}\left[\frac{\mathrm{HP}_{K[\cF_1]}(t)}{(1-t)^{s-2}}+\frac{\mathrm{HP}_{K[\cF_2]}(t)}{(1-t)^{r-1}}\right]$$
Finally we obtain our claims arguing as in the last part of the previous result (or also, for instance, as in Corollary~\ref{HilbertCsimple}).
\end{proof}

\subsection{$\cB_1$ and $\cB_2$ contain at least three cells}\label{ladder3}

Suppose that $\cB_1=[B_1,B]$ consists of the cells $B_1,\ldots, B_r,B$, $r\geq 2$, and $\cB_2=[A,A_r]$ of the cells $A,A_1,\ldots,A_s$, $s\geq2$. We denote  the upper and lower left corners of $A$ by $a,c$ respectively, the upper and lower right corners of $A$ by $b,d$ respectively, the left and right lower corners of $B$ by $f,g$ respectively, the upper and lower right corners of $A_i$ by $a_i,b_i$ respectively for $i\in[s]$, the lower and upper left corners of $B_i$ by $c_i,d_i$ respectively for $i\in [r]$.
Considering our assumption on the ladder at the beginning of Section \ref{Section: H-P series of prime closed path polyominoes having no L-configuration} and the fact that $\cP$ has not any $L$-configuration, we have that $c_1,c_2\notin V(\cP)\setminus V(\cB_1)$. The described arrangement is summarized in Figure \ref{Figura:Esempio ladder con B1 e B2 di lunghezza tre}. \\
For our purpose we need to introduce the following related polyominoes:
\begin{itemize}
	\item $\cK_1=\cP\backslash [B_1,B]$;
	\item $\cK_2=\cP\backslash ([A,A_s]\cup \{B,B_r\})$;
	\item $\cK_3=\cP\backslash ([B_1,B]\cup \{A\})$;
	\item $\cK_4=\cP\backslash \{A,B,A_1,B_r\}$.
\end{itemize}

 \begin{figure}[h!]
 	\centering
 \includegraphics[scale=0.9]{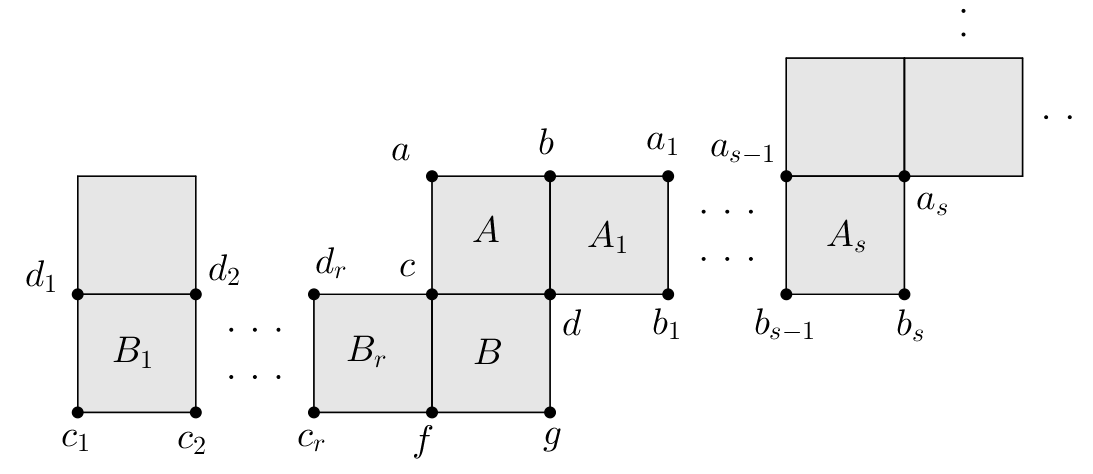}
 \caption{}
 \label{Figura:Esempio ladder con B1 e B2 di lunghezza tre}
 \end{figure}
 
\begin{lemma}\label{Lemma: colum + prodotti tensoriali - caso ladder B1 e B2 tre celle}
 Let $\cP$ be a closed path polyomino having a ladder of at least three steps satisfying the previous assumptions. Then the following hold:
 \begin{enumerate}[(1)]
 	\item $S_{\cP}/(I_{\cP}:x_g)\cong K[\cP]$;
 	\item $S_{\cP}/((I_{\cP},x_g):x_d)\cong K[\cK_1]\otimes_K K[x_{d_3},\dots,x_{d_r}]$;
 	\item $S_{\cP}/((I_{\cP},x_g,x_d):x_b)\cong K[\cK_2]\otimes_K K[x_{a},x_{b},x_{a_1},\dots,x_{a_{s-2}}]$;
 	\item $S_{\cP}/((I_{\cP},x_g,x_d,x_b):x_f)\cong S_{\cP}/(I_{\cP},x_g,x_d,x_b)$;
 	\item $S_{\cP}/((I_{\cP},x_g,x_d,x_b,x_f):x_c)\cong K[\cK_3]\otimes_K K[x_{d_3},\dots,x_{d_r}]$;
 	\item $S_{\cP}/((I_{\cP},x_g,x_d,x_b,x_f,x_c):x_a)\cong K[\cK_1]\otimes_K K[x_{a},x_{a_1},\dots,x_{a_{s-2}}]$;
 	\item $S_{\cP}/(I_{\cP},x_g,x_d,x_b,x_f,x_c,x_a)\cong K[\cK_4]$.
 \end{enumerate}  
\end{lemma}
\begin{proof}
	To prove the isomorphisms in the statements $(1)-(7)$, it is enough to prove the following equalities:
	\begin{enumerate}[(1)]
	\item $I_{\cP}:x_g=I_{\cP}$;
	\item $(I_{\cP},x_g):x_d=I_{\cK_1}+(x_f,x_g)+\sum_{i=1}^{r}(x_{c_i})$;
	\item $(I_{\cP},x_g,x_d):x_b=I_{\cK_2}+(x_f,x_g,x_d,x_c)+\sum_{i=1}^{s}(x_{b_i})$;
	\item $(I_{\cP},x_g,x_d,x_b):x_f=(I_{\cP},x_g,x_d,x_b)$;
	\item $(I_{\cP},x_g,x_d,x_b,x_f):x_c=(I_{\cK_1},x_b,x_d)+(x_g,x_f)+\sum_{i=1}^{r}(x_{c_i})$;
	\item $(I_{\cP},x_g,x_d,x_b,x_f,x_c):x_a=I_{\cK_2}+(x_f,x_g,x_d,x_c,x_b)+\sum_{i=1}^{s}(x_{b_i})$;
	\item $(I_{\cP},x_g,x_d,x_b,x_f,x_c,x_a)=(I_{\cK_4},x_g,x_d,x_b,x_f,x_c,x_a)$. 
	\end{enumerate}
In particular the equality (1) is trivial since $I_{\cP}$ is prime, $(2)$, $(3)$ and $(6)$, together with the related claims, can be proved as done in Lemma \ref{Lemma: column+prodotti tensoriali - primo caso}. The equality $(5)$ and its related claim follow as in Lemma~\ref{Lemma: column+prodotti tensoriali - secondo caso}, considering also that $S_{\cK_1}/(I_{\cK_1},x_b,x_d)\cong K[\cK_3]$ and arguing as in Lemma~\ref{P'_2}. We obtain the equality $(7)$ and its related claim	as for Lemma~\ref{isomorfismoP4}.\\
Finally, in  order to show $4)$, we prove that $(I_{\cP},x_g,x_d,x_b)$ is a prime ideal in $S_{\cP}$. Let $\{V_i\}_{i\in I}$ be the set of the maximal edge intervals of $\cP$ and $\{H_j\}_{j\in J}$ be the set of the maximal horizontal edge intervals of $\cP$. Let $\{v_i\}_{i\in I}$ and $\{h_j\}_{j\in J}$ be the set of the variables associated respectively to $\{V_i\}_{i\in I}$ and $\{H_j\}_{j\in J}$. Let $w$ be another variable and set $\cI=\{f,c,d,g,b_1,\dots,b_s\}\subset V(\cP)$. We consider the following ring homomorphism $$\phi: S_{\cP} \longrightarrow  K[\{v_i,h_j,w\}:i\in I,j\in J]$$ defined by $\phi(x_{ij})=v_ih_jw^k$, where $(i,j)\in V_i\cap H_j$, $k=0$ if $(i,j)\notin \cI$, and $k=1$ if $(i,j)\in \cI$. From \cite[Theorem 5.2]{Cisto_Navarra_closed_path} we have $I_{\cP}=\ker \phi$. Let $i'\in I$ such that $V_{i'}$ is the maximal edge interval of $\cP$ containing $b$, $d$ and $g$. We define $\psi: S_{\cP}\rightarrow K[\{v_i,h_j,w\}:i\in I\backslash\{i'\},j\in J]$ as $\psi(x_v)=\phi(x_v)$ if $v\in V(\cP)\backslash\{b,d,g\}$, and $\psi(x_b)=\psi(x_d)=\psi(x_g)=0$. It is not difficult to check that $(I_{\cP},x_d,x_b,x_g)\subseteq \ker\psi$. Let $f\in \ker \psi$. We can write $f=\tilde{f}+\beta x_b+ \delta x_d + \gamma x_g$ where $\beta, \delta, \gamma \in S_{\cP}$ and $x_b,x_d,x_g$ are not variables of $\tilde{f}$. Since $\psi(f)=0$, we have $\phi(\tilde{f})=0$, so $\tilde{f}\in \ker\phi=I_{\cP}$. Hence $S_{\cP}/(I_{\cP},x_b,x_d,x_g)\cong \mathrm{Im}(\psi)$, that is a domain since it is the subring of a domain. So $(I_{\cP},x_b,x_d,x_g)$ is prime in $S_{\cP}$. 
\end{proof}

\begin{remark}\rm
 If we suppose that $\cB_2$ has just two cells (so $s=1$), then $(I_{\cP},x_g,x_d,x_b)$ is not prime. In fact, set $b=a_0$, denote the cell adjacent to $A_1$ by $C$, and let $b,p$ and $q,a_1$ be respectively the diagonal and anti-diagonal corners of $C$. Observe that in such a case $x_qx_{a_1}\in (I_{\cP},x_g,x_d,x_b)$ but $x_q,x_{a_1}\notin (I_{\cP},x_g,x_d,x_b)$.
\end{remark}

\begin{theorem}\label{Teorema: serie di Hilbert - ladder con B1 e B2 di tre}
Let $\cP$ be a closed path polyomino having a ladder of at least three steps satisfying the assumptions at the beginning of Subsection \ref{ladder3}. Then $$\mathrm{HP}_{K[\cP]}(t)=\frac{h_{K[\cK_4]}(t)+t\big[h_{K[\cK_1]}(t)+2\cdot h_{K[\cK_2]}(t)+h_{K[\cK_3]}(t)\big]}{(1-t)^{\vert V(\cP)\vert -\mathrm{rank}\,\cP}}$$
In particular $K[\cP]$ has Krull dimension $\vert V(\cP)\vert -\mathrm{rank}\,\cP$.
\label{Hilbert_series_ladder}
\end{theorem}
\begin{proof}
   It follows from Lemma \ref{Lemma: colum + prodotti tensoriali - caso ladder B1 e B2 tre celle} considering the suitable exact sequences and arguing as done in Theorem \ref{Teorema: serie di Hilbert - primo caso} and Corollary \ref{HilbertCsimple}. 
\end{proof}

\section{Rook polynomial and Gorenstein property}

\noindent Let $\cP$ be a polyomino. A \textit{$k$-rook configuration} in $\cP$ is a configuration of $k$ rooks which are arranged in $\cP$ in non-attacking positions. 
\begin{figure}[h]
	\centering
	\includegraphics[scale=0.8]{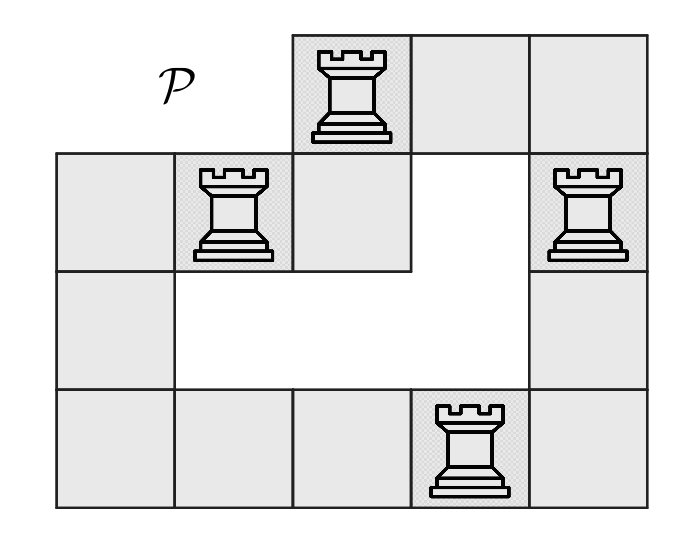}
	\caption{An example of a $4$-rook configuration in $\cP$.}
	\label{Figura:esempio rook configuration}
\end{figure}

\noindent The rook number $r(\cP)$ is the maximum number of rooks which can be placed in $\cP$ in non-attacking positions. We denote by $\cR(\cP,k)$ the set of all $k$-rook configurations in $\cP$ and we set $r_k=\vert \cR(\cP,k)\vert $ for all $k\in \{0,\dots,r(\cP)\}$, conventionally $r_0=1$. The \textit{rook polynomial} of $\cP$ is the polynomial $r_{\cP}(t)=\sum_{k=0}^{r(\cP)}r_kt^k \in \mathbb{Z}[t]$. \\
\noindent We recall that a polyomino is \emph{thin} if it does not contain the square tetromino, that is the square consisting of four cells. In \cite{Trento3} the authors prove that if $\cP$ is a simple thin polyomino then $h_{K[\cP]}(t)=r_{\cP}(t)$ and, in particular, $\mathrm{reg} K[\cP]=r(\cP)$ (see \cite[Theorem 3.12]{Trento3}). Now we show how the rook polynomial is related to Hilbert-Poincar\'e series of the polyominoes considered in this work.


\begin{proposition}\label{Proposizione: Proprietà grado del rook polynomial I caso}
	Let $\cP$ be a $(\cL,\cC)$-polyomino. Let $\cP_1,\cP_2,\cP_3,\cP_4$ be the polyominoes in Section 3. Then:
\begin{enumerate}[(1)]
	\item $r(\cP_1)=r(\cP_2)=r(\cP)-1$;
	\item $r(\cP_3)=r(\cP)-2$;
	\item $r(\cP)-2\leq r(\cP_4)\leq r(\cP)$.
\end{enumerate}
\end{proposition}
\begin{proof}
 $(1)$  Let $\cP_1= \cP\backslash [A,A_r]$. Once we fix a rook in a cell of $[A_1,\dots,A_{r-1}]$, we cannot place another rook in $[A,A_r]$ in non-attacking position in $\cP$, so $r(\cP_1)=r(\cP)-1$. In a similar way it can be showed that $r(\cP_2)=r(\cP)-1$.\\
 $(2)$ It follows by similar previous arguments on the intervals $[A,A_r]$ and $[A,B_s]$.\\
 $(3)$ Since $\cP=\cP_4\cup [A,A_1]\cup [A,B_1]$, it is obvious that $r(\cP_4)\leq r(\cP)$. Moreover, $\cP_4=\cP_3\cup [A_2,A_r]\cup [B_2,B_s]$, so $r(\cP_3)\leq r(\cP_4)$, that is $r(\cP)-2\leq r(\cP_4)$. In particular, observe that if $r,s>3$ then $r(\cP_4)=r(\cP)$, if either $r=3$ or $s=3$ then $r(\cP_4)=r(\cP)-1$, and if $r,s=3$ then $r(\cP_4)=r(\cP)-2$.
\end{proof}

\begin{theorem}\label{Teorema: h-polinomio = rook polinomio caso L-conf}
	Let $\cP$ be a $(\cL,\cC)$-polyomino. Suppose that $\cC$ is a simple thin polyomino. Then $h_{K[\cP]}(t)$ is the rook polynomial of $\cP$. Moreover $\mathrm{reg}(K[\cP])=r(\cP)$.
\end{theorem}
\begin{proof}
	It is known that $h_{K[\cP]}(t)=h_{K[\cP_4]}(t)+t\big[h_{K[\cP_1]}(t)+h_{K[\cP_2]}(t)+(1-t)h_{K[\cP_3]}(t)\big].$
	We denote by $r_{\cP_j}(t)=\sum_{k=0}^{r(\cP_j)}r_k^{(j)}t^k$ the rook polynomial of $\cP_j$.
	Since $\cC$ is a simple thin polyomino, then $\cP_1$, $\cP_2$, $\cP_3$ and $\cP_4$ are simple thin polyominoes, so $h_{K[\cP_j]}(t)=r_{\cP_j}(t)$ for $j\in\{1,2,3,4\}$. By Proposition \ref{Proposizione: Proprietà grado del rook polynomial I caso} we have $\deg h_{K[\cP]}=r(\cP)$. Then $$h_{K[\cP]}(t)=\sum_{k=0}^{r(\cP)}[r_k^{(4)}+r_{k-1}^{(1)}+r_{k-1}^{(2)}+r_{k-1}^{(3)}-r_{k-2}^{(3)}]t^k,$$
	where we set $r_{-1}^{(j)},r_{-2}^{(3)},r_{r(\cP)-1}^{(3)},r_k^{(4)}$ equal to $0$, for all $j\in\{1,2,3\}$ and for $k\geq r(\cP_4)$. \\
	We want to prove that $r_k^{(4)}+r_{k-1}^{(1)}+r_{k-1}^{(2)}+r_{k-1}^{(3)}-r_{k-2}^{(3)}$ is exactly the number of ways in which $k$ rooks can be placed in $\cP$ in non-attacking positions, for all $k\in\{0,\dots,r(\cP)\}$. Fix $k\in\{0,\dots,r(\cP)\}$. Observe that:
	\begin{enumerate}[(1)]
		\item $r_k^{(4)}$ can be viewed as the number of $k$-rook configurations in $\cP$ such that no rook is placed on $A,A_1$ and $B_1$.
		\item Assume that a rook $\mathcal{T}$ is placed in $A_1$. Then we cannot place any rook on a cell of $[A,A_r]$, so $r_{k-1}^{(1)}$ is the number of all $(k-1)$-rook configurations in $\cP_1$. Hence $r_{k-1}^{(1)}$ is the number of all $k$-rook configurations in $\cP$ such that a rook is on $A_1$. Observe that there are some $k$-rook configurations in $\cP$ in which a rook $\mathcal{T}'\neq \mathcal{T}$ is on $B_1$. Paraphrasing, note that $r_{k-1}^{(1)}$ is the number of all $k$-rook configurations in $\cP$ such that $\mathcal{T}$ is on $A_1$ and $\mathcal{T}'$ is not on $B_1$ plus those ones where $\mathcal{T}$ is on $A_1$ and $\mathcal{T}'$ is on $B_1$.
		\item Assume that a rook $\mathcal{T}$ is placed in $B_1$. Arguing as before, $r_{k-1}^{(1)}$ is the number of all $k$-rook configurations in $\cP$ such that $\mathcal{T}$ is on $B_1$ and $\mathcal{T}'$ is not on $A_1$ plus those ones where $\mathcal{T}$ is on $B_1$ and $\mathcal{T}'$ is on $A_1$. 
		\item Assume that a rook is placed on $A$. Then we cannot place any rook on a cell of $[A,A_r]\cup[A,B_s]$, so $r_{k-1}^{(3)}$ is the number of all $(k-1)$-rook configurations in $\cP_3$, that is the number of $k$-rook configurations in $\cP$ such that a rook is placed on $A$.
		\item  Fix a rook $\mathcal{T}$ in $A_1$ and another one $\mathcal{T}'$ in $B_1$. Then we cannot place any rook on a cell of $[A,A_r]\cup [A,B_s]$, so $r_{k-2}^{(3)}$ is the number of all $(k-2)$-rook configurations in $\cP_3$. Hence $r_{k-2}^{(3)}$ is the number of all $k$-rook configurations in $\cP$ such that a rook is on $A_1$ and another is on $B_1$.
	\end{enumerate} 
	From $(1),(2),(3),(4)$ and $(5)$ it follows that $r_k^{(4)}+r_{k-1}^{(1)}+r_{k-1}^{(2)}+r_{k-1}^{(3)}-r_{k-2}^{(3)}$ is the number of $k$-rook configurations in $\cP$.
\end{proof}


\begin{proposition}
Let $\cP$ be a closed path satisfying the conditions in Subsection~\ref{ladder3}. Then $h_{K[\cP]}(t)$ is the rook polynomial of $\cP$ and $\mathrm{reg}(K[\cP])=r(\cP)$.
\label{prop1}
\end{proposition}
\begin{proof}
It can be proved by similar arguments as in Proposition \ref{Proposizione: Proprietà grado del rook polynomial I caso} that the rook numbers of $\cK_1$, $\cK_2$, $\cK_3$ and $\cK_4$ satisfy the following:
		\begin{enumerate}[(1)]
			\item $r(\cK_1)=r(\cP)-1$;
			\item $r(\cP)-2\leq r(\cK_2)\leq r(\cP)-1$;
			\item $r(\cK_3)=r(\cP)-1$;
			\item $r(\cP)-2\leq r(\cK_4)\leq r(\cP)$.
		\end{enumerate}
		We denote by $r_{\cK_j}(t)=\sum_{k=0}^{r(\cK_j)}r_k^{(j)}t^k$ the rook polynomial of $\cK_j$, for $j=1,2,3,4$. Observe that
    $\cK_1$, $\cK_2$, $\cK_3$ and $\cK_4$ are simple thin polyominoes, so $h_{K[\cK_j]}(t)=r_{\cK_j}(t)$ for $j\in\{1,2,3,4\}$, and by the above formulas and Theorem~\ref{Hilbert_series_ladder} we have $\deg h_{K[\cP]}=r(\cP)$. Moreover $$h_{K[\cP]}(t)=\sum_{k=0}^{r(\cP)}[r_k^{(4)}+r_{k-1}^{(1)}+2r_{k-1}^{(2)}+r_{k-1}^{(3)}]t^k,$$
	where we set $r_{-1}^{(j)},r_k^{(4)},r_l^{(2)}$ equal to $0$, for all $j\in\{1,2,3\}$, for $k\geq r(\cK_4)$ and $l\geq r(\cK_2)$. \\
	Similarly as done in Theorem \ref{Teorema: h-polinomio = rook polinomio caso L-conf}, we have that $r_k^{(4)}+r_{k-1}^{(1)}+2r_{k-1}^{(2)}+r_{k-1}^{(3)}$ is the number of $k$-rook configurations in $\cP$, for all $k\in\{0,\dots,r(\cP)\}$. In fact, let $k\in\{0,\dots,r(\cP)\}$. Observe that:
	\begin{enumerate}[(1)]
		\item $r_k^{(4)}$ is the number of $k$-rook configurations in $\cP$ such that no rook is placed on $A,A_1$, $B$ and $B_r$.
		\item Fix a rook $\mathcal{T}$ on $B_r$. Then $r_{k-1}^{(1)}$ is the number of all $k$-rook configurations in $\cP$ such that $\mathcal{T}$ is on $B_r$. Observe that among these configurations, there are some $k$-rook configurations in which $\mathcal{T}'\neq \mathcal{T}$ is placed either in $A$ or in $A_1$.
		\item Fix a rook $\mathcal{T}$ in $B$. Then $r_{k-1}^{(3)}$ is the number of all $k$-rook configurations in $\cP$ such that $\mathcal{T}$ is on $B$. As before, among these configurations there are some $k$-rook configurations in which $\mathcal{T}'\neq \mathcal{T}$ is placed in $A_1$.
	\item Assume that a rook is placed in $A$ (resp. $A_1$). Then $r_{k-1}^{(2)}$ is the number of all $k$-rook configurations in $\cP$ such that $\mathcal{T}$ is on $A$ (resp. $A_1$), and no rook is on a cell of $[A,A_s]\cup \{B,B_r\}$.
	\end{enumerate}
	From $(1),(2),(3)$ and $(4)$ we have the desired conclusion.
\end{proof}

\noindent In order to complete the study of closed path polyominoes having no $L$-configuration, it remains to consider 1-Configuration and 2-Configuration introduced in Subsection~\ref{ladder2}. For such cases we mention only the analogous result, omitting the proof since the arguments are similar.

\begin{proposition}\label{prop2}
Let $\cP$ be a closed path which satisfies the conditions of 1-Configuration or 2-Configurations in Subsection~\ref{ladder2}. Then $h_{K[\cP]}(t)$ is the rook polynomial of $\cP$ and $\mathrm{reg}(K[\cP])=r(\cP)$.
\end{proposition}

\noindent Observing that a closed path having an $L$-configuration is an $(\cL,\cC)$-polyomino with $\cC$ a path of cells and gathering all the results above, we obtain the following general result.

\begin{theorem}
Let $\cP$ be a closed path having no zig-zag walks, equivalently having an $L$-configuration or a ladder of three steps. Then:
\begin{enumerate}[(1)]
	\item  $K[\cP]$ is a normal Cohen-Macaulay domain of Krull dimension $\vert V(\cP)\vert -\mathrm{rank}\,\cP$;
	\item $h_{K[\cP]}(t)$ is the rook polynomial of $\cP$ and $\mathrm{reg}(K[\cP])=r(\cP)$.
\end{enumerate} 
\label{ultimo}
\end{theorem}

 \noindent At this point we are ready to provide the condition for the Gorenstein property of a closed path polyomino having no zig-zag walks. Recall the definition of \textit{S-property} given in \cite{Trento3} for a thin polyomino.
 \begin{definition}\rm
 	Let $\cP$ be a thin polyomino. A cell $C$ is called \textit{single} if there exists a unique maximal interval of $\cP$ containing $C$. $\cP$ has the \textit{S-property} if every maximal interval of $\cP$ has only one single cell.
 \end{definition}

\noindent Observe that if $\cP$ is a closed path polyomino, then $\cP$ has the S-property if and only if every maximal block of $\cP$ contains exactly three cells. 
 
\begin{theorem}
	Let $\cP$ be a closed path having no zig-zag walks. The following are equivalent:
	\begin{enumerate}[(1)]
		\item $\cP$ has the $S$-property;
		\item $K[\cP]$ is Gorenstein.
	\end{enumerate}
	\end{theorem}
	\begin{proof}
		If $\cP$ has no zig-zag walks, then  $K[\cP]$ is a normal Cohen-Macaulay domain of Krull dimension $\vert V(\cP)\vert -\mathrm{rank}\,\cP$ 
		and $h_{K[\cP]}(t)=r_{\cP}(t)=\sum_{k=0}^{s}r_kt^k$, where $s=r(\cP)$. In such a case it is known (\cite{Stanley}) that $K[\cP]$ is Gorenstein if and only if $r_i=r_{s-i}$ for all $i=0,\dots,s$.\\
		$(1)\Rightarrow (2)$. Suppose that $\cP$ has the $S$-property. Fix $i\in\{0,1,\dots,r(\cP)\}$ and prove that $r_i=r_{s-i}$. Since $\cP$ has the $S$-property, $\cP$ consists of maximal cell intervals of rank three. If $i=0$ then it is trivial that $r_0=r_s=1$. Assume $i\in[s-1]$. It is not restrictive to consider a part of $\cP$ arranged as in Figure \ref{Figura:Dimostrazione Gorenstein L conf}.
		\begin{figure}[h]
			\centering
			\includegraphics[scale=0.7]{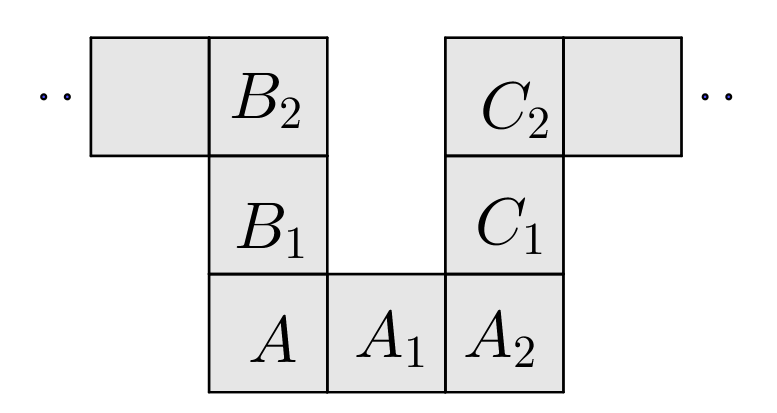}
			\caption{}
			\label{Figura:Dimostrazione Gorenstein L conf}
		\end{figure} 
	
		\noindent Define $\cP_1=\cP\backslash\{A,A_1,A_2\}$, $\cP_2=\cP\backslash\{A,A_1,A_2,C_1,C_2\}$ and $\cP_3=\cP\backslash\{A,A_1,A_2,B_1,B_2\}$.
		We denote by $r_{\cP_j}(t)=\sum_{k=0}^{r(\cP_j)}r_k^{(j)}t^k$ the rook polynomial of $\cP_j$. Observe that $r(\cP_1)=r(\cP)-1=s-1$ and $r(\cP_2)=r(\cP_3)=r(\cP)-2=s-2$. By similar arguments as in Theorem \ref{Teorema: h-polinomio = rook polinomio caso L-conf}, it is easy to prove that $r_{k}=r_k^{(1)}+r_{k-1}^{(1)}+r_{k-1}^{(2)}+r_{k-1}^{(3)}$ for all $k\in\{1,\dots,s\}$. Then
		\begin{align*}
			&r_{s-i}=r_{s-i}^{(1)}+r_{s-i-1}^{(1)}+r_{s-i-1}^{(2)}+r_{s-i-1}^{(3)}=r_{(s-1)-(i-1)}^{(1)}+\\
			&+r_{(s-1)-i}^{(1)}+r_{(s-2)-(i-1)}^{(2)}+r_{(s-2)-(i-1)}^{(3)}.
		\end{align*} 
	Since $\cP_1$, $\cP_2$ and $\cP_3$ are simple thin polyominoes having the $S$-property, then by Theorem 4.2 of \cite{Trento3} we have: $r_{(s-1)-(i-1)}^{(1)}=r_{i-1}^{(1)}$, $r_{(s-1)-i}^{(1)}=r_{i}^{(1)}$, $r_{(s-1)-(i-2)}^{(2)}=r_{i-1}^{(2)}$ and $r_{(s-2)-(i-1)}^{(3)}=r_{i-1}^{(3)}$. Hence
	\begin{align*}
		&r_{s-i}=r_{(s-1)-(i-1)}^{(1)}+r_{(s-1)-i}^{(1)}+r_{(s-2)-(i-1)}^{(2)}+r_{(s-2)-(i-1)}^{(3)}=r_{i-1}^{(1)}+\\
		&+r_{i}^{(1)}+r_{i-1}^{(2)}+r_{i-1}^{(3)}=r_i.
	\end{align*}
	$(2)\Rightarrow (1)$. Assume that $K[\cP]$ is Gorenstein, that is $r_i=r_{s-i}$ for all $i=0,\dots,s$. We prove that $\cP$ has the $S$-property. First of all, we observe that all the ranks of the maximal intervals of $\cP$ cannot be greater than or equal to four. In fact, if there exists a maximal interval $I=[A,B]$ with $\mathrm{rank}\, I\geq 4$, then we can consider two distinct cells $C,D\in I\backslash \{A,B\}$. Hence we can obtain an $s$-rook configuration in $\cP$ with a rook in $C$ and another one with a rook in $D$, so $r_{s}\geq 2>r_0=1$, that is a contradiction. In addition, in such a case, we can suppose that $\cP$ has an $L$-configuration, otherwise it is not difficult to see that $\cP$ has a subpolyomino as in Figure~\ref{Figura:pentomino}, and arguing as in the proof of the case $b)\Rightarrow c)$ (hypothesis (2)) of \cite[Theorem 4.2]{Trento3}, then $K[\cP]$ is not Gorenstein. So, let $\{A,A_1,A_2,B_1,B_2\}$ be an $L$-configuration of $\cP$, as in Figure \ref{Figura:Dimostrazione Gorenstein L conf}. Consider $\cP'=\cP\backslash\{A,A_1,A_2\}$, which is a simple thin polyomino. Let $r_{\cP'}(t)=\sum_{k=0}^{s'}r_k't^k$ be the rook polynomial of $\cP'$, where $s'=r(\cP)-1$. We prove that $\cP'$ has the $S$-property. Suppose that $\cP'$ has not the $S$-property so from the case $b)\Rightarrow c)$ of \cite[Theorem 4.2]{Trento3} it follows that either $r'_{s'}>1$ or $r'_{s'-1}>\mathrm{rank}\, \cP'$. Both cases lead to a contradiction with $r_s=1$ or $r_{s-1}=\mathrm{rank}\, \cP$. By similar arguments we can prove that $\cP''=\cP\backslash \{A,B_1,B_2\}$ is a simple thin polyomino having the $S$-property. Since $\cP'$ and $\cP''$ have the $S$-property, it follows trivially that also $\cP$ has the $S$-property. 
	\end{proof}
	
\begin{remark}\rm
If $\cQ$ is the weakly closed path introduced Subsection~\ref{ladder2}, then $K[\cQ]$ is a normal Cohen-Macaulay domain (Remark~\ref{RemarkQ}). From Theorem~\ref{HilbertQ} we obtain that $\mathrm{HP}_{K[\cQ]}(t)=\frac{h(t)}{(1-t)^{\vert V(\cP)\vert -\mathrm{rank}\,\cP}}$ with $h(t)=h_{K[\cP]}(t)-h_{K[\cQ_1]}(t)$. We know that $h_{K[\cP]}(t)$ and $h_{K[\cQ_1]}(t)$ are the rook polynomials respectively of $\cP$ and $\cQ_1$. Since $\cQ_1$ is contained in $\cP$, then $h(1)>0$. So the Krull dimension of $K[\cQ]$ is $\vert V(\cP)\vert -\mathrm{rank}\,\cP=\vert V(\cQ)\vert -\mathrm{rank}\,\cQ$. Moreover, by the same arguments adopted in this work, we obtain that $h(t)$ is the rook polynomial of $\cQ$ and $K[\cQ]$ is Gorenstein if and only if $\cQ$ has the $S$-property. 
\end{remark}
\vspace{0.1cm}
\begin{flushleft}
	\textbf{Concluding remarks.}\\
\end{flushleft}

\noindent In the existing literature, the Hilbert-Poincar\'e series and the Gorenstein property of polyomino ideals have been discussed only for a few classes of simple polyominoes. In this work, we provide
some results in this direction for a class of non-simple polyominoes, known as the closed paths without zig-zag walks. The coordinate rings of these polyominoes are known to be Cohen-Macaulay. In the light of our results, we propose the following questions:
\begin{enumerate}[(1)]
\item Is it possible to characterize the Gorenstein property for $(\cL,\cC)$-polyominoes based on the choice of $\cC$?
\item In this work, the Hilbert-Poincar\'e series of closed path polyominoes without zig-zag walks is studied. Is it possible to provide similar results for the Hilbert-Poincar\'e series of a closed path
polyomino with a zig-zag walk?\\
\end{enumerate} 

\noindent \textbf{Disclosure of Potential Conflicts of Interest:} The authors declare that there is no conflict of interest.

	\end{document}